\documentclass{article}
\usepackage{ijcai19}

\usepackage{times}
\usepackage{soul}
\usepackage{url}
\usepackage[hidelinks]{hyperref}
\usepackage[utf8]{inputenc}
\usepackage[small]{caption}
\usepackage{graphicx}
\usepackage{amsmath}
\usepackage{booktabs}
\usepackage{algorithm}
\usepackage{algorithmic}
\usepackage{amsfonts}
\usepackage{amsthm}
\usepackage{amsmath}

\usepackage{appendix}

\usepackage{multirow}

\newtheorem{theorem}{Theorem}
\newtheorem{lemma}{Lemma}

\newtheorem{definition}{Definition}
\newtheorem{assumption}{Assumption}
\newtheorem{remark}{Remark}

\usepackage{color}
\usepackage{hyperref}

\usepackage{graphicx}
\usepackage[tight,footnotesize]{subfigure}

\title{ Zeroth-Order Stochastic Alternating Direction Method of Multipliers \\ for Nonconvex Nonsmooth Optimization }

\author{Feihu Huang$^{1}$, Shangqian Gao$^1$, Songcan Chen$^{2,3}$ {\normalfont and} Heng Huang$^{1,4}$\thanks{Corresponding Author.} \\ \affiliations
$^1$ Department of Electrical \& Computer Engineering, University of Pittsburgh, USA\\
$^2$ College of Computer Science \& Technology, Nanjing University of Aeronautics and Astronautics \\
$^3$ MIIT Key Laboratory of Pattern Analysis \& Machine Intelligence, China
\\
$^4$ JD Finance America Corporation\\
feh23@pitt.edu, \ shg84@pitt.edu, \ s.chen@nuaa.edu.cn, \ heng.huang@pitt.edu
}

\begin{document}

\maketitle
\begin{abstract}
Alternating direction method of multipliers (ADMM) is a popular optimization tool
for the composite and constrained problems in machine learning.
However, in many machine learning problems such as black-box learning and bandit feedback,
ADMM could fail because the explicit gradients of these problems are difficult
or even infeasible to obtain. Zeroth-order (gradient-free) methods can effectively solve these problems due to that the objective function values
are only required in the optimization. Recently, though there exist a few zeroth-order ADMM methods,
they build on the convexity of objective function.
Clearly, these existing zeroth-order methods are limited in many applications.
In the paper, thus, we propose a class of fast zeroth-order stochastic ADMM methods
(\emph{i.e.}, ZO-SVRG-ADMM and ZO-SAGA-ADMM) for solving nonconvex problems with multiple nonsmooth penalties, based on the
coordinate smoothing gradient estimator.
Moreover, we prove that both the ZO-SVRG-ADMM and ZO-SAGA-ADMM have convergence rate of $O(1/T)$, where $T$ denotes
the number of iterations.
In particular, our methods not only reach the best convergence rate of $O(1/T)$ for the nonconvex optimization,
but also are able to effectively solve many complex machine learning problems with multiple regularized penalties and constraints.
Finally, we conduct the experiments of black-box binary classification and structured adversarial attack on black-box deep neural network
to validate the efficiency of our algorithms.
\end{abstract}

\section{Introduction}
Alternating direction method of multipliers (ADMM \cite{gabay1976dual,boyd2011distributed}) is a popular optimization tool
for solving the composite and constrained problems in machine learning.
In particular, ADMM can efficiently optimize some problems with complicated structure regularization
such as the graph-guided fused lasso \cite{kim2009multivariate},
which is too complicated for the other popular optimization methods
such as proximal gradient methods \cite{beck2009fast}.
For the large-scale optimization,
the stochastic ADMM method \cite{ouyang2013stochastic} has been proposed.
Recently, some faster stochastic ADMM methods \cite{suzuki2014stochastic,zheng2016fast}
have been proposed by using the variance reduced (VR) techniques such as the SVRG \cite{johnson2013accelerating}.
In fact, ADMM is also highly successful in solving various nonconvex problems
such as training deep neural networks \cite{Taylor2016Training}.
Thus, some fast nonconvex stochastic ADMM methods have been developed in \cite{huang2016stochastic,Huang2019fasterst}.

Currently, most of the ADMM methods
need to  compute the gradients
of objective functions over each iteration.
However, in many machine learning problems,
the explicit expression of gradient for objective function
is difficult or infeasible to obtain.
For example, in black-box situations, only prediction results (\emph{i.e.}, function
values) are provided \cite{chen2017zoo,liu2018zeroth}.
In bandit settings \cite{agarwal2010optimal},
player only receives the partial feedback in terms of loss function values,
so it is impossible to obtain expressive gradient of the loss function.
Clearly, the classic optimization methods, based on the first-order gradient or second-order information, are not competent to these problems.
Recently, the zeroth-order optimization methods \cite{duchi2015optimal,Nesterov2017RandomGM}
are developed by only using the function values in the optimization.
\begin{table*}
  \centering
  \begin{tabular}{c|c|c|c|c}
  \hline
 \textbf{Algorithm} & \textbf{Reference} & \textbf{Gradient Estimator} &  \textbf{Problem}  & \textbf{Convergence Rate} \\ \hline
  ZOO-ADMM & \cite{liu2018admm} &GauSGE &  C(S) + C(NS) & $O(\sqrt{1/T})$  \\ \hline
  ZO-GADM & \cite{gao2018information} & UniSGE & C(S) + C(NS)&  $O(\sqrt{1/T})$  \\ \hline
  RSPGF  &\cite{ghadimi2016mini} & GauSGE &NC(S) + C(NS)  & $O(\sqrt{1/T})$ \\ \hline
  ZO-ProxSVRG & \multirow{2}*{\cite{Huang2019faster}} & \multirow{2}*{CooSGE} & \multirow{2}*{NC(S) + C(NS)} & \multirow{2}*{$O(1/T)$} \\
  ZO-ProxSAGA &  & & &   \\ \hline
  ZO-SVRG-ADMM & \multirow{2}*{Ours} & \multirow{2}*{CooSGE} & \multirow{2}*{NC(S) + C(mNS)} & \multirow{2}*{$O(1/T)$} \\
  ZO-SAGA-ADMM &  & & &   \\ \hline
  \end{tabular}
  \caption{  Convergence properties comparison of the zeroth-order ADMM algorithms and other ones.
  C, NC, S, NS and mNS are the abbreviations of convex,
    non-convex, smooth, non-smooth and the sum of multiple non-smooth functions, respectively.
    $T$ is the whole iteration number.
    Gaussian Smoothing Gradient Estimator (GauSGE),
    Uniform Smoothing Gradient Estimator (UniSGE) and Coordinate Smoothing Gradient Estimator (CooSGE). }\label{tab:1}
\end{table*}

In the paper, we focus on using the zeroth-order methods to solve the following nonconvex nonsmooth problem:
\begin{align} \label{eq:1}
 \min_{x,\{y_j\}_{j=1}^k} & F(x,y_{[k]}) := \frac{1}{n}\sum_{i=1}^n f_i(x) + \sum_{j=1}^{k}\psi_j(y_j)  \\
 \mbox{s.t.} \ & Ax + \sum_{j=1}^kB_jy_j =c, \nonumber
\end{align}
where $A\in \mathbb{R}^{p\times d}$, $B_j\in \mathbb{R}^{p\times q}$ for all $j\in [k],\ k \geq 1$,
$f(x)=\frac{1}{n}\sum_{i=1}^n f_i(x): \mathbb{R}^d\rightarrow \mathbb{R}$
is a \emph{nonconvex} and black-box function, and each $\psi_j(y_j): \mathbb{R}^q\rightarrow \mathbb{R}$ is a convex and
\emph{nonsmooth} function.
In machine learning, function $f(x)$ can be used for the empirical loss,
$\sum_{j=1}^{k}\psi_j(y_j)$ for multiple structure penalties (\emph{e.g.}, sparse + group sparse), and the constraint for encoding
the structure pattern of model parameters such as graph structure.
Due to the flexibility in splitting the objective function into loss $f(x)$ and each penalty $\psi_j(y_j)$,
ADMM is an efficient method to solve the obove problem. However,
in the problem \eqref{eq:1}, we only access the objective values
rather than the whole explicit function $F(x,y_{[k]})$,
thus the classic ADMM methods are unsuitable for the problem \eqref{eq:1}.

Recently, \cite{gao2018information,liu2018admm} proposed the zeroth-order stochastic ADMM methods,
which only use the objective values to optimize. However,
these zeroth-order ADMM-based methods build on the convexity of objective function. Clearly, these methods
are limited in many nonconvex problems such as adversarial attack on black-box deep neural network (DNN).
At the same time, due to that the problem \eqref{eq:1} includes multiple nonsmooth regularization functions and an equality constraint,
the existing zeroth-order algorithms \cite{liu2018zeroth,Huang2019faster}
are not suitable for solving this problem.

In the paper, thus, we propose a class of
fast zeroth-order stochastic ADMM methods (\emph{i.e.}, ZO-SVRG-ADMM and ZO-SAGA-ADMM) to solve the problem \eqref{eq:1} based on the coordinate
smoothing gradient estimator \cite{liu2018zeroth}.
In particular, the ZO-SVRG-ADMM and ZO-SAGA-ADMM methods build on
the SVRG \cite{johnson2013accelerating} and SAGA \cite{defazio2014saga}, respectively.
Moreover, we study the convergence properties of the proposed methods.
Table \ref{tab:1} shows the convergence properties
of the proposed methods and other related ones.
\subsection{Challenges and Contributions}
Although both SVRG and SAGA show good performances in the first-order and second-order methods,
applying these techniques to the nonconvex zeroth-order ADMM method
is \emph{not trivial}. There exists at least two main \textbf{challenges}:
\begin{itemize}
 \item Due to failure of the Fej\'{e}r monotonicity of iteration, the convergence analysis of the nonconvex ADMM is generally quite difficult \cite{wang2015convergence}.
       With using the inexact zeroth-order estimated gradient, this difficulty becomes greater in the nonconvex ADMM methods.
 \item To guarantee convergence of our zeroth-order ADMM methods, we need to design a new effective \emph{Lyapunov} function, which can not follow
       the existing nonconvex (stochastic) ADMM methods \cite{jiang2019structured,huang2016stochastic}.
\end{itemize}
Thus, we carefully establish the \emph{Lyapunov} functions in the following theoretical analysis to
ensure convergence of the proposed methods.
In summary, our major \textbf{contributions} are given below:
\begin{itemize}
\item[1)] We propose a class of fast zeroth-order stochastic ADMM methods (\emph{i.e.}, ZO-SVRG-ADMM and ZO-SAGA-ADMM) to solve the problem \eqref{eq:1}.
\item[2)] We prove that both the ZO-SVRG-ADMM and ZO-SAGA-ADMM have convergence rate of $O(\frac{1}{T})$ for nonconvex nonsmooth optimization.
          In particular, our methods not only reach the existing best convergence rate $O(\frac{1}{T})$ for the nonconvex optimization,
          but also are able to effectively solve many machine learning problems with multiple complex regularized penalties.
\item[3)] Extensive experiments conducted on black-box classification and structured adversarial attack on black-box DNNs
          validate efficiency of the proposed algorithms.
\end{itemize}
\section{Related Works}
Zeroth-order (gradient-free) optimization is a powerful optimization tool for solving many
machine learning problems, where the gradient
of objective function is not available or computationally prohibitive.
Recently, the zeroth-order optimization methods are widely applied and studied.
For example, zeroth-order optimization methods have been applied to bandit feedback analysis \cite{agarwal2010optimal} and
black-box attacks on DNNs \cite{chen2017zoo,liu2018zeroth}.
\cite{Nesterov2017RandomGM} have proposed several random zeroth-order
methods based on the Gaussian smoothing gradient estimator.
To deal with the nonsmooth regularization, \cite{gao2018information,liu2018admm} have proposed
the zeroth-order online/stochastic ADMM-based methods.

So far, the above algorithms mainly build on the convexity of problems.
In fact, the zeroth-order methods are also highly successful in solving various nonconvex problems such as
adversarial attack to black-box DNNs \cite{liu2018zeroth}.
Thus, \cite{ghadimi2013stochastic,liu2018zeroth,Gu2018faster}
have begun to study the zeroth-order stochastic methods for the nonconvex optimization.
To deal with the nonsmooth regularization, \cite{ghadimi2016mini,Huang2019faster} have proposed
some non-convex zeroth-order proximal stochastic gradient methods.
However, these methods still are not well competent to some complex machine learning problems
such as a task of structured adversarial attack to the black-box DNNs, which is described in the following experiment.
\subsection{Notations}
Let $y_{[k]}= \{y_1,\cdots,y_k\}$ and $y_{[j:k]}= \{y_{j},\cdots,y_k\}$ for $j\in[k]$.
Given a positive definite matrix $G$, $\|x\|^2_G = x^TGx$;
$\sigma_{\max}(G)$ and $\sigma_{\min}(G)$
denote the largest and smallest eigenvalues of $G$, respectively,
and $\kappa_G = \frac{\sigma_{\max}(G)}{\sigma_{\min}(G)}$.
$\sigma^A_{\max}$ and $\sigma^A_{\min}$ denote the largest and smallest eigenvalues of matrix $A^TA$.
\section{Preliminaries}
In the section, we begin with restating a standard $\epsilon$-approximate stationary point of the problem \eqref{eq:1},
as in \cite{jiang2019structured,Huang2019fasterst}.
\begin{definition} \label{def:1}
Given $\epsilon>0$, the point $(x^*,y_{[k]}^*,\lambda^*)$ is said to be an
$\epsilon$-approximate stationary point of the problems \eqref{eq:1},
if it holds that
\begin{align}
 \mathbb{E}\big[ \mbox{dist}(0,\partial L(x^*,y_{[k]}^*,\lambda^*))^2 \big] \leq \epsilon,
\end{align}
where $L(x,y_{[k]},\lambda)=f(x) + \sum_{j=1}^k\psi_j(y_j) - \langle \lambda, Ax+\sum_{j=1}^kB_jy_j-c\rangle$,
\begin{align}
   \partial L(x,y_{[k]},\lambda) = \left [ \begin{matrix}
     \nabla_x L(x,y_{[k]},\lambda) \\
     \partial_{y_1} L(x,y_{[k]},\lambda) \\
     \cdots \\
     \partial_{y_k} L(x,y_{[k]},\lambda) \\
     -Ax - \sum_{j=1}^kB_jy_j+c
 \end{matrix}
 \right ], \nonumber
\end{align}
$\mbox{dist}(0,\partial L)=\inf_{L'\in \partial L} \|0-L'\|.$
\end{definition}

Next, we make some mild assumptions regarding problem \eqref{eq:1} as follows:
\begin{assumption}
Each function $f_i(x)$ is $L$-smooth for $\forall i \in \{1,2,\cdots,n\}$ such that
\begin{align}
\|\nabla f_i(x)-\nabla f_i(y)\| \leq L \|x - y\|, \ \forall x,y \in \mathbb{R}^d, \nonumber
\end{align}
which is equivalent to
\begin{align}
f_i(x) \leq f_i(y) + \nabla f_i(y)^T(x-y) + \frac{L}{2}\|x-y\|^2.  \nonumber
\end{align}
\end{assumption}
\begin{assumption}
Full gradient of loss function $f(x)$ is bounded, i.e., there exists a constant $\delta >0$ such that for all $x$,
it follows that $\|\nabla f(x)\|^2 \leq \delta^2$.
\end{assumption}
\begin{assumption}
$f(x) $ and $\psi_j(y_j)$ for all $j\in [k]$ are all lower bounded, and denote
$f^* = \inf_x f(x)$ and $\psi_j^* = \inf_y \psi_j(y)$ for $j\in [k]$.
\end{assumption}
\begin{assumption}
$A$ is a full row or column rank matrix.
\end{assumption}
Assumption 1 has been commonly used in the convergence analysis of
nonconvex algorithms \cite{ghadimi2016mini}.
Assumption 2 is widely used for
stochastic gradient-based and ADMM-type methods \cite{boyd2011distributed}.
Assumptions 3 and 4 are usually used in the convergence analysis of ADMM methods \cite{jiang2019structured,huang2016stochastic,Huang2019fasterst}.
Without loss of generality, we will use the full column rank of matrix $A$ in the rest of this paper.
\section{ Fast Zeroth-Order Stochastic ADMMs }
\begin{algorithm}[htb]
   \caption{ Nonconvex ZO-SVRG-ADMM Algorithm }
   \label{alg:1}
\begin{algorithmic}[1]
   \STATE {\bfseries Input:} $b$, $m$, $T$, $S=\lceil T/m\rceil$, $\eta>0$ and $\rho>0$;
   \STATE {\bfseries Initialize:} $x_0^1$, $y^{0,1}_j$ for $j\in [k]$ and $\lambda_0^1$;
   \FOR {$s=1,2,\cdots,S$}
   \STATE{} $\tilde{x}^{s+1}=x_0^{s+1}$, $\hat{\nabla} f(\tilde{x}^{s})=\frac{1}{n}\sum_{i=1}^n\hat{\nabla} f_i(\tilde{x}^{s})$;
   \FOR {$t=0,1,\cdots,m-1$}
   \STATE{} Uniformly randomly pick a mini-batch $\mathcal{I}_t$ (with replacement) from $\{1,2,\cdots,n\}$, and $|\mathcal{I}_t|=b$ ;
   \STATE{}  Using \eqref{eq:5} to estimate stochastic gradient
             $\hat{g}_{t}^{s} = \hat{\nabla} f_{\mathcal{I}_t}(x_{t}^{s})-\hat{\nabla} f_{\mathcal{I}_t}(\tilde{x}^s)+\hat{\nabla} f(\tilde{x}^s)$;
   \STATE{}  $y_j^{s,t+1} \!= \! \arg\min_{y_j} \! \big\{ \mathcal {L}_{\rho}(x_t^s,y_{[j-1]}^{s,t+1},y_j,y_{[j+1:k]}^{s,t},\lambda_t^s)$
             $ \quad + \frac{1}{2}\|y_j-y_j^{s,t}\|^2_{H_j} \big\} $, for all $j\in [k]$;
   \STATE{}  $x^s_{t+1}= \arg\min_x \hat{\mathcal {L}}_{\rho}\big( x,y_{[k]}^{s,t+1},\lambda^s_t, \hat{g}_{t}^{s} \big)$;
   \STATE{}  $\lambda^s_{t+1} = \lambda^s_t-\rho(Ax^s_{t+1} + \sum_{j=1}^kB_jy_j^{s,t+1}-c)$;
   \ENDFOR
   \STATE{} $x_0^{s+1}=x_{m}^{s}$, $y_j^{s+1,0}=y_j^{s,m}$ for $j\in [k]$, $\lambda_0^{s+1}=\lambda_{m}^{s}$;
   \ENDFOR
   \STATE {\bfseries Output:} $\{x,y_{[k]},\lambda\}$ chosen at random uniformly from $\{(x_{t}^s,y_{[k]}^{s,t},\lambda_t^s)_{t=1}^{m}\}_{s=1}^S$.
\end{algorithmic}
\end{algorithm}

In this section, we propose a class of zeroth-order stochastic ADMM methods to solve
the problem \eqref{eq:1}. First, we define an augmented
Lagrangian function of the problem \eqref{eq:1}:
\begin{align}
 \mathcal {L}_{\rho}(x,y_{[k]},\lambda) & \!= f(x) + \sum_{j=1}^k\psi_j(y_j) \!-\! \langle\lambda, Ax \!+\! \sum_{j=1}^kB_jy_j-c\rangle \nonumber \\
 &  + \frac{\rho}{2} \|Ax + \sum_{j=1}^kB_jy_j - c\|^2,
\end{align}
where $\lambda\in \mathbb{R}^{p}$ and $\rho >0$ denotes the dual variable and penalty parameter, respectively.

\begin{algorithm}[htb]
   \caption{ Nonconvex ZO-SAGA-ADMM Algorithm }
   \label{alg:2}
\begin{algorithmic}[1]
   \STATE {\bfseries Input:} $b$, $T$, $\eta>0$ and $\rho>0$;
   \STATE {\bfseries Initialize:} $z_i^0=x_0$ for $i\in \{1,2,\cdots,n\}$, $\hat{\phi}_0=\frac{1}{n}\sum_{i=1}^n\nabla f_i(z^0_i)$, $y^{0}_j$ for $j\in [k]$ and $\lambda_0$;
   \FOR {$t=0,1,\cdots,T-1$}
   \STATE{} Uniformly randomly pick a mini-batch $\mathcal{I}_t$ (with replacement) from $\{1,2,\cdots,n\}$, and $|\mathcal{I}_t|=b$ ;
   \STATE{} Using \eqref{eq:5} to estimate stochastic gradient $\hat{g}_{t} \!=\! \frac{1}{b}\sum_{i_t\in \mathcal{I}_t}
             \big(\hat{\nabla} f_{i_t}(x_{t})\!-\!\hat{\nabla} f_{i_t}(z^t_{i_t}) \big)\!+\!\hat{\phi}_t$
            with $\hat{\phi}_t\!=\!\frac{1}{n}\sum_{i=1}^n\hat{\nabla} f_i(z^t_i)$;
   \STATE{}  $y_j^{t+1} = \arg\min_{y_j} \big\{ \mathcal {L}_{\rho}(x_t,y_{[j-1]}^{t+1},y_j,y_{[j+1:k]}^{t},\lambda_t)
              + \frac{1}{2}\|y_j-y_j^t\|^2_{H_j} \big\}$, for all $j\in [k]$;
   \STATE{}  $x_{t+1}= \arg\min_x \hat{\mathcal {L}}_{\rho}\big( x,y_{[k]}^{t+1},\lambda_t, \hat{g}_{t} \big)$;
   \STATE{}  $\lambda_{t+1} = \lambda_t-\rho(Ax_{t+1} + \sum_{j=1}^kB_jy_j^{t+1}-c)$;
   \STATE{} $z^{t+1}_{i_t}= x_{t}$ for $i_t \in \mathcal{I}_t$ and $z_{i_t}^{t+1}=z^t_{i_t}$ for $i_t \not\in \mathcal{I}_t$;
   \STATE{} $\hat{\phi}_{t+1}=\hat{\phi}_t-\frac{1}{n}\sum_{i_t\in \mathcal{I}_t} \big(\hat{\nabla} f_{i_t}(z^t_{i_t})-\hat{\nabla} f_{i_t}(z^{t+1}_{i_t})\big)$;
   \ENDFOR
   \STATE {\bfseries Output:} $\{x,y_{[k]},\lambda\}$ chosen at random uniformly from $\{x_{t},y_{[k]}^{t},\lambda_t\}_{t=1}^{T}$.
\end{algorithmic}
\end{algorithm}
In the problem \eqref{eq:1}, the explicit expression of objective function $f_i(x)$ is not available, and
only the function value of $f_i(x)$ is available. To avoid
computing explicit gradient, thus, we use the coordinate smoothing gradient estimator
\cite{liu2018zeroth} to estimate gradients: for $i\in [n]$,
\begin{align} \label{eq:5}
 \hat{\nabla} f_i(x) = \sum_{j=1}^d \frac{1}{2\mu_j}\big(
 f_i(x+\mu_j e_j)-f_i(x-\mu_j e_j)\big)e_j,
\end{align}
where $\mu_j$ is a coordinate-wise smoothing parameter, and $e_j$ is
a standard basis vector with 1 at its $j$-th coordinate, and 0
otherwise.

Based on the above estimated gradients, we propose a zeroth-order
ADMM (ZO-ADMM) method to solve the problem \eqref{eq:1} by executing the following iterations,
for $t \!=\! 1,2,\cdots$
\begin{equation*}
\left\{
\begin{aligned}
& y_j^{t+1} = \arg\min_{y_j} \big\{ \mathcal {L}_{\rho}(x_t,y_{[j-1]}^{t+1},y_j,y_{[j+1:k]}^{t},\lambda_t) \\
& \quad \qquad \qquad \qquad + \frac{1}{2}\|y_j-y_j^t\|^2_{H_j} \big\}, \ \forall j\in [k] \\
& x_{t+1} = \arg\min_x \hat{\mathcal {L}}_{\rho}(x,y_{t+1},\lambda_t,\hat{\nabla} f(x)) \\
& \lambda_{t+1} = \lambda_t - \rho(Ax_{t+1}+By_{t+1}-c),
\end{aligned} \right.
\end{equation*}
where the term $\frac{1}{2}\|y_j-y_j^t\|^2_{H_j}$ with $H_j\succ 0$ to linearize the term $\|Ax + \sum_{j=1}^kB_jy_j-c\|^2$.
Here, due to using the inexact zeroth-order gradient to update $x$, we define an approximate function over $x_t$ as follows:
\begin{align}
\small
 & \hat{\mathcal {L}}_{\rho}\big(x,y_{[k]}^{t+1},\lambda_t,\hat{\nabla} f(x)\big) \!=\!
 f(x_t) \!+\! \hat{\nabla}f(x)^T(x-x_t) \nonumber \\
 & \!+\! \frac{1}{2\eta}\|x\!-\!x_t\|^2_G \!+\! \sum_{j=1}^k \psi_j(y_j^{t+1}) \!-\! \lambda_t^T(Ax \!+\! \sum_{j=1}^k B_jy_j^{t+1}\!-\!c) \nonumber \\
 & \!+\! \frac{\rho}{2}\|Ax \!+\! \sum_{j=1}^kB_jy_j^{t+1}\!-\!c\|^2,
\end{align}
where $G\succ0$, $\hat{\nabla} f(x)$ is the zeroth-order gradient and $\eta>0$ is a step size.
Considering the matrix $A^TA$ is large, set $G = r I - \rho \eta A^TA \succ I$ with $r > \rho \eta
\sigma_{\max}(A^TA) + 1 $ to linearize the term $\|Ax + \sum_{j=1}^kB_jy_j^{t+1}-c\|^2$.
In the problem \eqref{eq:1}, not only the noisy
gradient of $f_i(x)$ is not available, but also the sample size $n$
is very large. Thus, we propose fast ZO-SVRG-ADMM and ZO-SAGA-ADMM to solve the problem \eqref{eq:1},
based on the SVRG and SAGA, respectively.

Algorithm \ref{alg:1} shows the algorithmic framework of ZO-SVRG-ADMM.
In Algorithm \ref{alg:1}, we use the estimated stochastic gradient
$\hat{g}_{t}^{s} = \hat{\nabla} f_{\mathcal{I}_t}(x_{t}^{s})-\hat{\nabla} f_{\mathcal{I}_t}(\tilde{x}^s)+\hat{\nabla} f(\tilde{x}^s)$
with $\hat{\nabla} f_{\mathcal{I}_t}(x^s_t) = \frac{1}{b}\sum_{i_t\in \mathcal{I}_t} \hat{\nabla} f_{i_t}(x^s_t)$.
We have $\mathbb{E}_{\mathcal{I}_t}[\hat{g}_{t}^{s}] = \hat{\nabla} f(x_{t}^{s}) \neq \nabla f(x_{t}^{s})$, \emph{i.e.}, this stochastic gradient
is a \textbf{biased} estimate of the true full gradient.
Although the SVRG has shown a great promise,
it relies upon the assumption that the stochastic
gradient is an \textbf{unbiased }estimate of true full gradient.
Thus, adapting the similar ideas of SVRG to zeroth-order ADMM optimization
is not a trivial task. To handle this challenge, we choose the appropriate step size
$\eta$, penalty parameter $\rho$ and smoothing parameter $\mu$ to guarantee the convergence of our algorithms,
which will be discussed in the following convergence analysis.

Algorithm \ref{alg:2} shows the algorithmic framework of ZO-SAGA-ADMM.
In Algorithm \ref{alg:2}, we use the estimated stochastic gradient $\hat{g}_{t} = \frac{1}{b} \sum_{i_t\in \mathcal{I}_t}
\big(\hat{\nabla} f_{i_t}(x_{t})-\hat{\nabla} f_{i_t}(z^t_{i_t}) \big) + \hat{\phi}_t$
with $\hat{\phi}_t=\frac{1}{n}\sum_{i=1}^n \hat{\nabla} f_i(z^t_i)$.
Similarly, we have $\mathbb{E}_{\mathcal{I}_t}[\hat{g}_{t}] = \hat{\nabla} f(x_t) \neq \nabla f(x_t)$.
\section{Convergence Analysis}
In this section, we will study the convergence properties of
the proposed algorithms (ZO-SVRG-ADMM and ZO-SAGA-ADMM).
\subsection{Convergence Analysis of ZO-SVRG-ADMM }
In this subsection, we analyze convergence properties of the ZO-SVRG-ADMM.

Given the sequence $\{(x^{s}_t,y_{[k]}^{s,t},\lambda^{s}_t)_{t=1}^m\}_{s=1}^S$ generated from Algorithm \ref{alg:1},
we define a \emph{Lyapunov} function:
 \begin{align}
 R^s_t \! = & \mathbb{E}\big[\mathcal{L}_{\rho} (x^s_t,y_{[k]}^{s,t},\lambda^s_t) \!+\! (\frac{3\sigma^2_{\max}(G)}{\sigma^A_{\min}\eta^2\rho} \!+\! \frac{9L^2}{\sigma^A_{\min}\rho})\|x^s_{t}\!-\!x^s_{t-1}\|^2 \nonumber \\
 & \quad + \frac{18 L^2d }{\sigma^A_{\min}\rho b}\|x^s_{t-1}-\tilde{x}^s\|^2 + c_t\|x^s_{t}-\tilde{x}^s\|^2\big], \nonumber
 \end{align}
 where the positive sequence $\{c_t\}$ satisfies
 \begin{equation*}
  c_t= \left\{
  \begin{aligned}
  & \frac{36 L^2d }{\sigma^A_{\min}\rho b} +
     \frac{2Ld}{b} + (1+\beta)c_{t+1}, \ 1 \leq t \leq m, \\
  & 0, \ t \geq m+1.
  \end{aligned}
  \right.\end{equation*}
Next, we definite a useful variable $\theta^s_{t} \!=\! \mathbb{E} \big[ \|x^s_{t+1}-x^s_{t}\|^2 + \|x^s_{t}-x^s_{t-1}\|^2 + \frac{d}{b}(\|x^s_{t}-\tilde{x}^s\|^2 + \|x^s_{t-1}-\tilde{x}^s\|^2 )
   + \sum_{j=1}^k \|y_j^{s,t}-y_j^{s,t+1}\|^2\big]$.
\begin{theorem} \label{th:1}
 Suppose the sequence $\{(x^{s}_t,y_{[k]}^{s,t},\lambda^{s}_t)_{t=1}^m\}_{s=1}^S$ is generated from Algorithm \ref{alg:1}. Let $m=n^{\frac{1}{3}}$,
 $b=d^{1-l} n^{\frac{2}{3}},\ l \in\{ 0,\frac{1}{2},1\}$, $\eta = \frac{\alpha\sigma_{\min}(G)}{9d^lL} \ (0 < \alpha \leq 1 )$ and
 $\rho = \frac{6\sqrt{71}\kappa_G d^l L}{\sigma^A_{\min}\alpha}$,
 then we have
 \begin{align}
\min_{s,t} \mathbb{E}\big[ \mbox{dist}(0,\partial L(x^s_t,y_{[k]}^{s,t},\lambda^s_t))^2\big] \leq O(\frac{\tilde{\nu} d^{2l}}{T}) + O(d^{2+2l}\mu^2), \nonumber
\end{align}
where $\tilde{\nu}=R^1_0 - R^*$, and $R^*$ is a lower bound of function $R^s_t$.
It follows that suppose the smoothing parameter $\mu$ and the whole iteration number $T=mS$ satisfy
 \begin{align}
  \frac{1}{\mu} = O \big( \frac{d^{1+l}}{\sqrt{\epsilon}} \big), \quad
   T = O\big( \frac{\tilde{\nu} d^{2l}}{\epsilon} \big), \nonumber
 \end{align}
 then $(x^{s^*}_{t^*},y_{[k]}^{s^*,t^*},\lambda^{s^*}_{t^*})$
 is an $\epsilon$-approximate stationary point of the problems \eqref{eq:1}, where $(t^*,s^*) = \mathop{\arg\min}_{t,s}\theta^s_{t}$.
\end{theorem}
\begin{remark}
Theorem \ref{th:1} shows that given $m=n^{\frac{1}{3}}$,
 $b=d^{1-l} n^{\frac{2}{3}},\ l \in\{ 0,\frac{1}{2},1\}$, $\eta = \frac{\alpha\sigma_{\min}(G)}{9d^lL} \ (0 < \alpha \leq 1 )$,
 $\rho = \frac{6\sqrt{71}\kappa_G d^l L}{\sigma^A_{\min}\alpha}$ and $\mu= O(\frac{1}{d\sqrt{T}})$,
 the ZO-SVRG-ADMM has convergence rate of $O(\frac{d^{2l}}{T})$. Specifically, when $1\leq d < n^{\frac{1}{3}}$, given $l=0$,
 the ZO-SVRG-ADMM has convergence rate of $O(\frac{1}{T})$; when $n^{\frac{1}{3}} \leq d < n^{\frac{2}{3}}$, given $l=\frac{1}{2}$,
 it has convergence rate of $O(\frac{\sqrt{d}}{T})$; when $n^{\frac{2}{3}} \leq d$, given $l=1$,
 it has convergence rate of $O(\frac{d}{T})$.
\end{remark}
\subsection{Convergence Analysis of ZO-SAGA-ADMM }
In this subsection, we provide the convergence analysis of
the ZO-SAGA-ADMM.

Given the sequence $\{x_t,y_{[k]}^{t},\lambda_t\}_{t=1}^T$ generated from Algorithm \ref{alg:2},
we define a \emph{Lyapunov} function
 \begin{align}
 & \Omega_t \! = \mathbb{E}\big[ \mathcal{L}_{\rho} (x_t,y_{[k]}^{t},\lambda_t) \!+\! (\frac{3\sigma^2_{\max}(G)}{\sigma^A_{\min}\rho\eta^2}
  \!+\! \frac{9L^2}{\sigma^A_{\min}\rho}) \|x_{t}\!-\!x_{t-1}\|^2  \nonumber \\
 &\quad + \frac{18 L^2d }{\sigma^A_{\min}\rho b}\frac{1}{n} \sum_{i=1}^n\|x_{t-1}-z^{t-1}_i\|^2 + c_t\frac{1}{n} \sum_{i=1}^n\|x_{t}-z^t_i\|^2 \big]. \nonumber
\end{align}
Here the positive sequence $\{c_t\}$ satisfies
 \begin{equation*}
 c_t \!= \left\{
  \begin{aligned}
  & \frac{36L^2d }{\sigma^A_{\min}\rho b} \!+ \frac{2Ld}{b} \!+ (1-\hat{p})(1+\beta)c_{t+1}, \ 0 \leq t \leq T-1, \\
  & 0, \ t \geq T,
  \end{aligned}
  \right.    \end{equation*}
 where $\hat{p}$ denotes probability of an index $i$ in $\mathcal{I}_t$.
Next, we definite a useful variable $\theta_{t} \!=\! \mathbb{E} \big[\|x_{t+1}\!-\!x_t\|^2 \!+\! \|x_t\!-\!x_{t-1}\|^2
\!+\! \frac{d}{bn}\sum^n_{i=1} (\|x_t\!-\!z^t_i\|^2 \!+\! \|x_{t-1}\!-\!z^{t-1}_i\|^2)\!+\! \sum_{j=1}^k \|y_j^{t}\!-\!y_j^{t+1}\|^2 \big]$.
\begin{theorem} \label{th:2}
Suppose the sequence $\{x_t,y_{[k]}^{t},\lambda_t\}_{t=1}^T$ is generated from Algorithm \ref{alg:2}. Let $b = n^{\frac{2}{3}}d^{\frac{1-l}{3}},\ l \in\{ 0,\frac{1}{2},1\}$,
$\eta = \frac{\alpha\sigma_{\min}(G)}{33d^lL} \ (0 < \alpha \leq 1)$ and $\rho = \frac{6\sqrt{791}\kappa_G d^l L}{\sigma^A_{\min}\alpha}$
then we have
\begin{align}
\min_{1\leq t \leq T} \mathbb{E}\big[ \mbox{dist}(0,\partial L(x_t,y_{[k]}^{t},\lambda_t))^2\big] \leq O(\frac{\tilde{\nu}d^{2l}}{T}) \!+ O(d^{2+2l}\mu^2), \nonumber
\end{align}
where $\tilde{\nu}=\Omega_0- \Omega^*$, and $\Omega^*$ is a lower bound of function $\Omega_t$.
It follows that suppose the parameters $\mu$ and $T$ satisfy
\begin{align}
 \frac{1}{\mu} = O \big( \frac{d^{1+l} }{\sqrt{\epsilon}} \big), \quad  T = O\big(\frac{\tilde{\nu}d^{2l}}{\epsilon} \big), \nonumber
\end{align}
then $(x_{t^*},y_{[k]}^{t^*},\lambda_{t^*})$ is an $\epsilon$-approximate stationary point of the problems \eqref{eq:1},
where $t^* = \mathop{\arg\min}_{ 1\leq t\leq T}\theta_{t}$.
\end{theorem}
\begin{remark}
Theorem \ref{th:2} shows that $b = n^{\frac{2}{3}}d^{\frac{1-l}{3}},\ l \in\{ 0,\frac{1}{2},1\}$,
$\eta = \frac{\alpha\sigma_{\min}(G)}{33d^lL} \ (0 < \alpha \leq 1)$, $\rho = \frac{6\sqrt{791}\kappa_G d^l L}{\sigma^A_{\min}\alpha}$
and $\mu= O(\frac{1}{d\sqrt{T}})$, the ZO-SAGA-ADMM has the $O(\frac{d^{2l}}{T})$ of convergence rate. Specifically, when $1\leq d < n$, given $l=0$,
 the ZO-SAGA-ADMM has convergence rate of $O(\frac{1}{T})$; when $n \leq d < n^2$, given $l=\frac{1}{2}$,
 it has convergence rate of $O(\frac{d}{T})$; when $n^2 \leq d$, given $l=1$,
 it has convergence rate of $O(\frac{d^2}{T})$.
\end{remark}
\begin{figure*}[!t]
\centering
\subfigure[20news]{\includegraphics[width=0.235\textwidth]{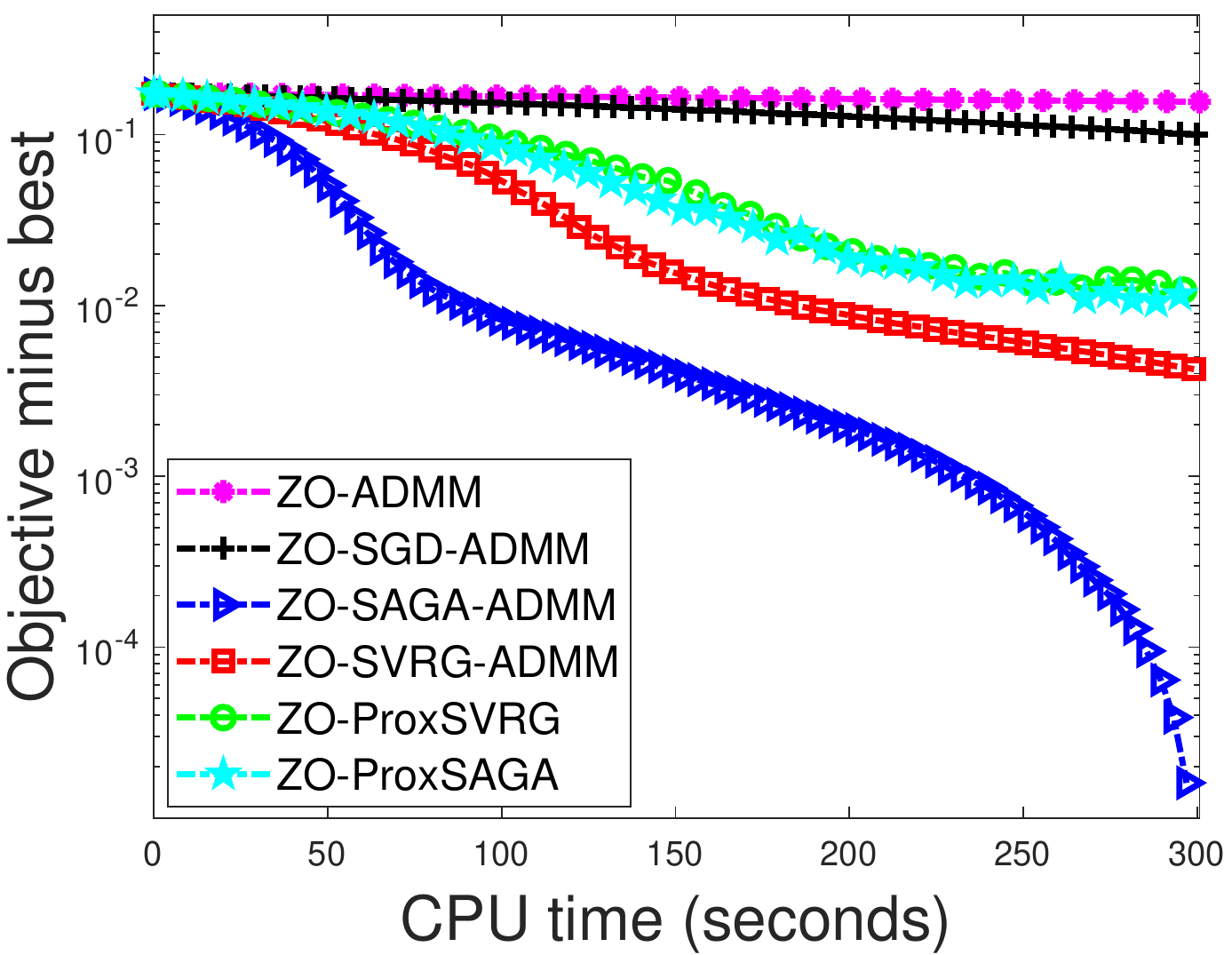}}
\subfigure[a9a]{\includegraphics[width=0.235\textwidth]{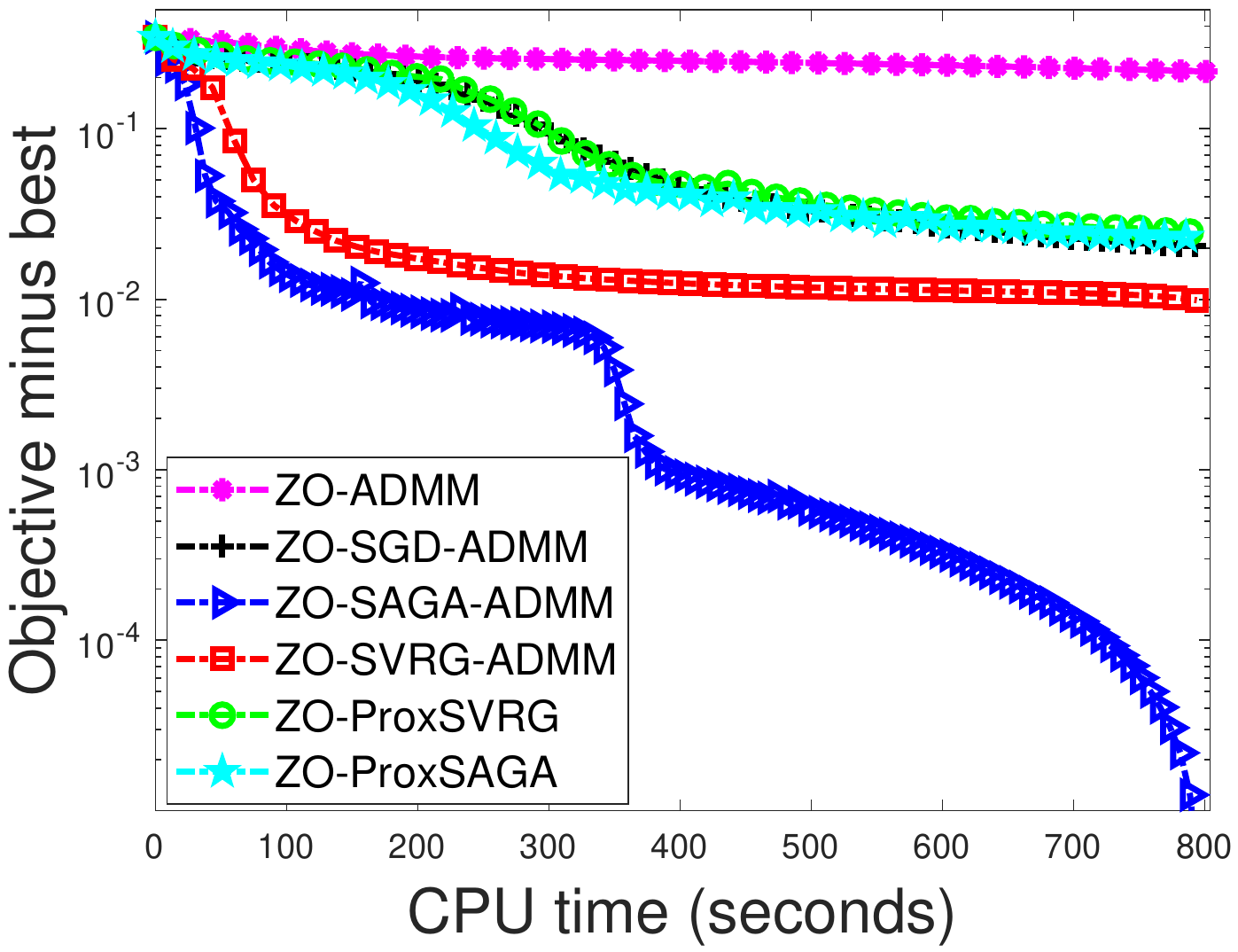}}
\subfigure[w8a]{\includegraphics[width=0.235\textwidth]{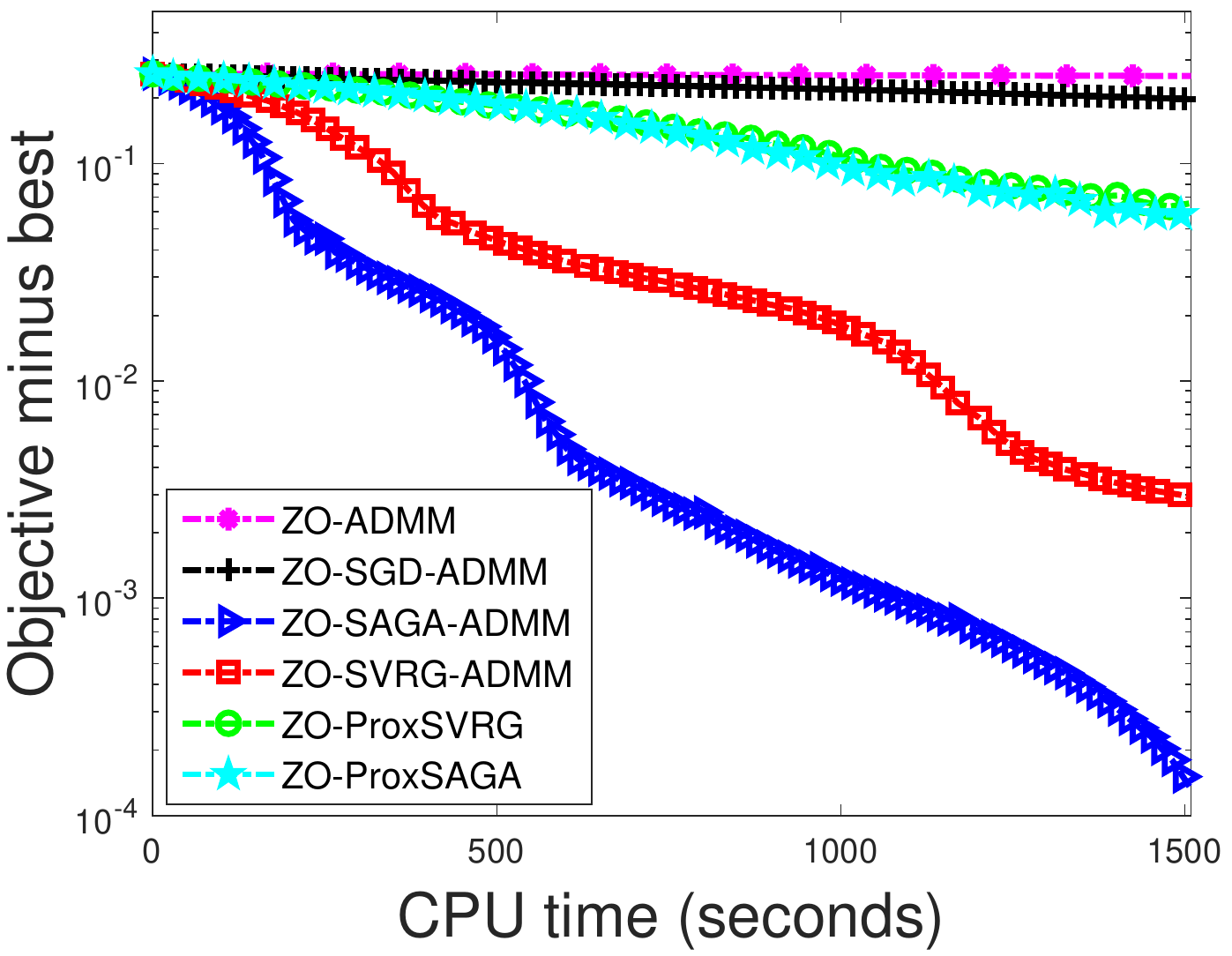}}
\subfigure[covtype.binary]{\includegraphics[width=0.235\textwidth]{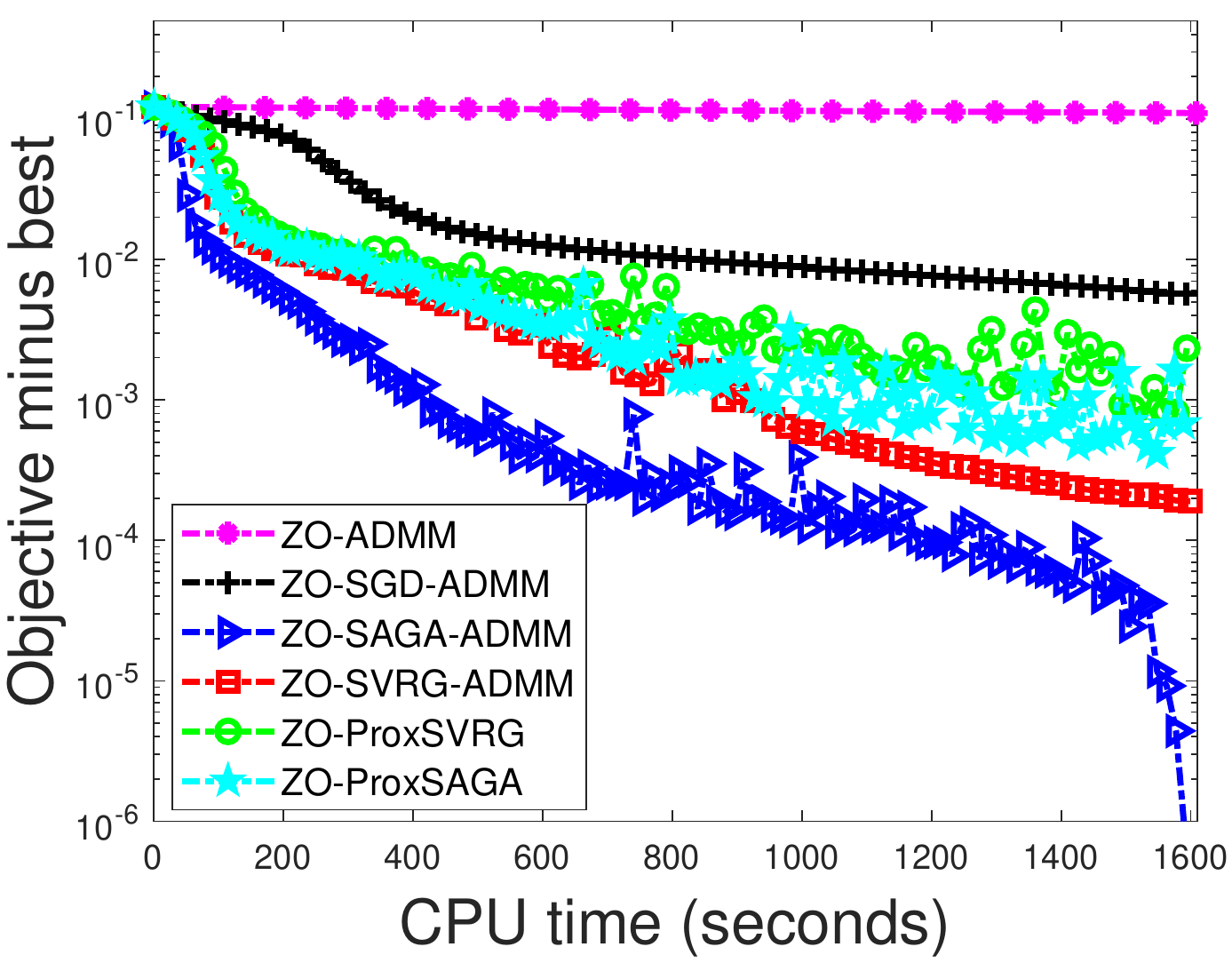}}
\caption{Objective value gaps \emph{versus} CPU time on benchmark datasets.}
\label{fig:1}
\end{figure*}
\section{Experiments}
In this section, we compare our algorithms
(ZO-SVRG-ADMM, ZO-SAGA-ADMM) with the ZO-ProxSVRG, ZO-ProxSAGA \cite{Huang2019faster}, the deterministic zeroth-order ADMM (ZO-ADMM),
and zeroth-order stochastic ADMM (ZO-SGD-ADMM) without variance reduction
on two applications: 1) robust black-box binary classification, and 2) structured adversarial
attacks on black-box DNNs.
\begin{table}[!h]
  \centering
  \begin{tabular}{c|c|c|c}
  \hline
  datasets & \#samples & \#features & \#classes \\ \hline
  \emph{20news} & 16,242 &  100 & 2 \\
  \emph{a9a}   & 32,561 & 123 & 2 \\
  \emph{w8a} & 64,700 & 300 & 2 \\
  \emph{covtype.binary} & 581,012 & 54 & 2\\
  \hline
  \end{tabular}
  \caption{Real Datasets for Black-Box Binary Classification }\label{tab:2}
\end{table}

\begin{figure*}[!t]
  \centering
  \includegraphics[width=0.75\textwidth]{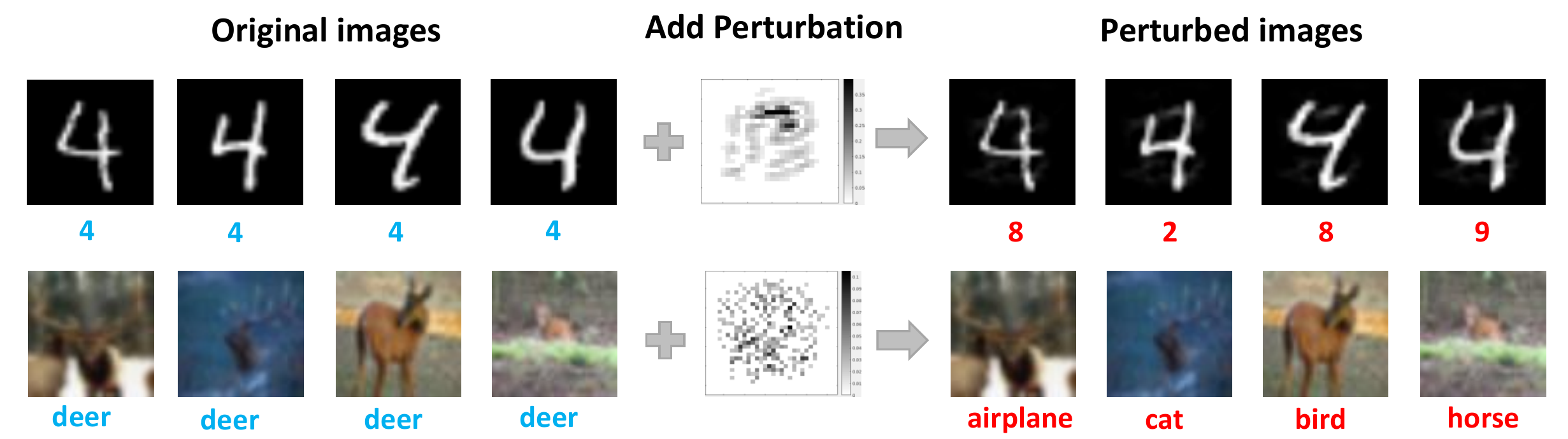}\\
  \caption{Group-sparsity perturbations are learned from MNIST and CIFAR-10 datasets. Blue and red labels denote the initial label,
  and the label after attack, respectively. }\label{fig:3}
\end{figure*}
\subsection{ Robust Black-Box Binary Classification }
In this subsection, we focus on a robust black-box binary classification task
with graph-guided fused lasso.
Given a set of training samples $(a_i,l_i)_{i=1}^n$,
where $a_i\in \mathbb{R}^d$ and $l_i \in \{-1,+1\}$,
we find the optimal parameter $x\in \mathbb{R}^d$ by solving the problem:
\begin{align} \label{eq:8}
 \min_{x\in \mathbb{R}^d} \frac{1}{n}\sum_{i=1}^n f_i(x) + \tau_1\|x\|_1 + \tau_2\|\hat{G}x\|_1,
\end{align}
where $f_i(x)$ is the black-box loss function,
that only returns the function value given an input.
Here, we specify the loss function
$f_i(x)= \frac{\sigma^2}{2}\big(1-\exp(-\frac{(l_i-a_i^Tx)^2}{\sigma^2})\big)$,
which is the \emph{nonconvex} robust correntropy
induced loss \cite{he2011maximum}.
Matrix $\hat{G}$ decodes the sparsity pattern of graph obtained by learning sparse Gaussian graphical model \cite{huang2015joint}.
In the experiment, we give mini-batch size $b=20$,
smoothing parameter $\mu=\frac{1}{d\sqrt{t}}$ and penalty parameters $\tau_1=\tau_2=10^{-5}$.

In the experiment, we use
some public real datasets\footnote{\emph{20news} is from \url{https://cs.nyu.edu/~roweis/data.html};
others are from \url{www.csie.ntu.edu.tw/~cjlin/libsvmtools/datasets/}.},
which are summarized in Table \ref{tab:2}. For each dataset, we use half of the samples as training data
and the rest as testing data. Figure \ref{fig:1} shows that the objective values
of our algorithms faster decrease than the
other algorithms, as the CPU time increases. In particular, our algorithms show better performances than
the zeroth-order proximal algorithms. It is relatively difficult that these zeroth-order proximal methods deal with
the nonsmooth penalties in the problem \eqref{eq:8}.
Thus, we have to use some iterative methods
to solve the proximal operator in these proximal methods.
\subsection{Structured Attacks on Black-Box DNNs }
In this subsection, we use our algorithms to generate adversarial examples to attack the pre-trained DNN models,
whose parameters are hidden from us and only its outputs are accessible.
Moreover, we consider an interesting problem: ``What possible
structures could adversarial perturbations have to fool black-box DNNs ?"
Thus, we use the zeroth-order algorithms to find an universal structured adversarial perturbation $x \in \mathbb{R}^d$
that could fool the samples $\{a_i\in \mathbb{R}^d , \ l_i \in \mathbb{N} \}_{i=1}^n$,
which can be regarded as the following problem:
\begin{align} \label{eq:9}
\min_{x \in \mathbb{R}^d} &\frac{1}{n}\sum\limits_{i=1}^n \max\big\{F_{l_i}(a_i+x) - \max\limits_{j\neq l_i} F_j(a_{i}+x), 0 \big\} \nonumber \\
 &  + \tau_1 \sum_{p=1}^P\sum_{q=1}^Q \|x_{\mathcal{G}_{p,q}}\|_2 + \tau_2 \|x \|_2^2 + \tau_3 h(x),
\end{align}
where $F(a)$ represents the final layer output before softmax of neural network, and $h(x)$ ensures the validness of created adversarial examples.
Specifically, $h(x) = 0$ if $a_i+x \in [0,1]^d$ for all $i\in [n]$ and $\|x\|_\infty \leq \epsilon$, otherwise $h(x) = \infty$.
Following \cite{xu2018structured}, we use the overlapping lasso to obtain structured perturbations.
Here, the overlapping groups $\{\mathcal{G}_{p,q}\},\ p=1,\cdots,P, \ q = 1,\cdots,Q$
generate from dividing an image into
sub-groups of pixels.

In the experiment, we use the pre-trained DNN models on MNIST and CIFAR-10 as the target black-box models, which can attain $99.4\%$ and $80.8\%$ test accuracy, respectively.
For MNIST, we select 20 samples from a target class and set batch size $b=4$; For CIFAR-10, we select 30 samples and set $b=5$.
In the experiment, we set $\mu=\frac{1}{d\sqrt{t}}$, where $d=28 \times 28$ and $d=3\times 32\times 32$ for MNIST and CIFAR-10, respectively.
At the same time, we set the parameters $\epsilon=0.4$, $\tau_1 = 1$, $\tau_2=2$ and $\tau_3=1$.
For both datasets, the kernel size for overlapping group lasso is set to $3 \times 3$ and the stride is one.

Figure \ref{fig:2} shows that attack losses (\emph{i.e.} the first term of the problem \eqref{eq:9}) of
our methods faster decrease than the other methods, as the number of iteration increases.
Figure \ref{fig:3} shows that our algorithms can learn some structure perturbations,
and can successfully attack the corresponding DNNs.
\begin{figure}[htbp]
\centering
\subfigure[\emph{MNIST}]{\includegraphics[width=0.236\textwidth]{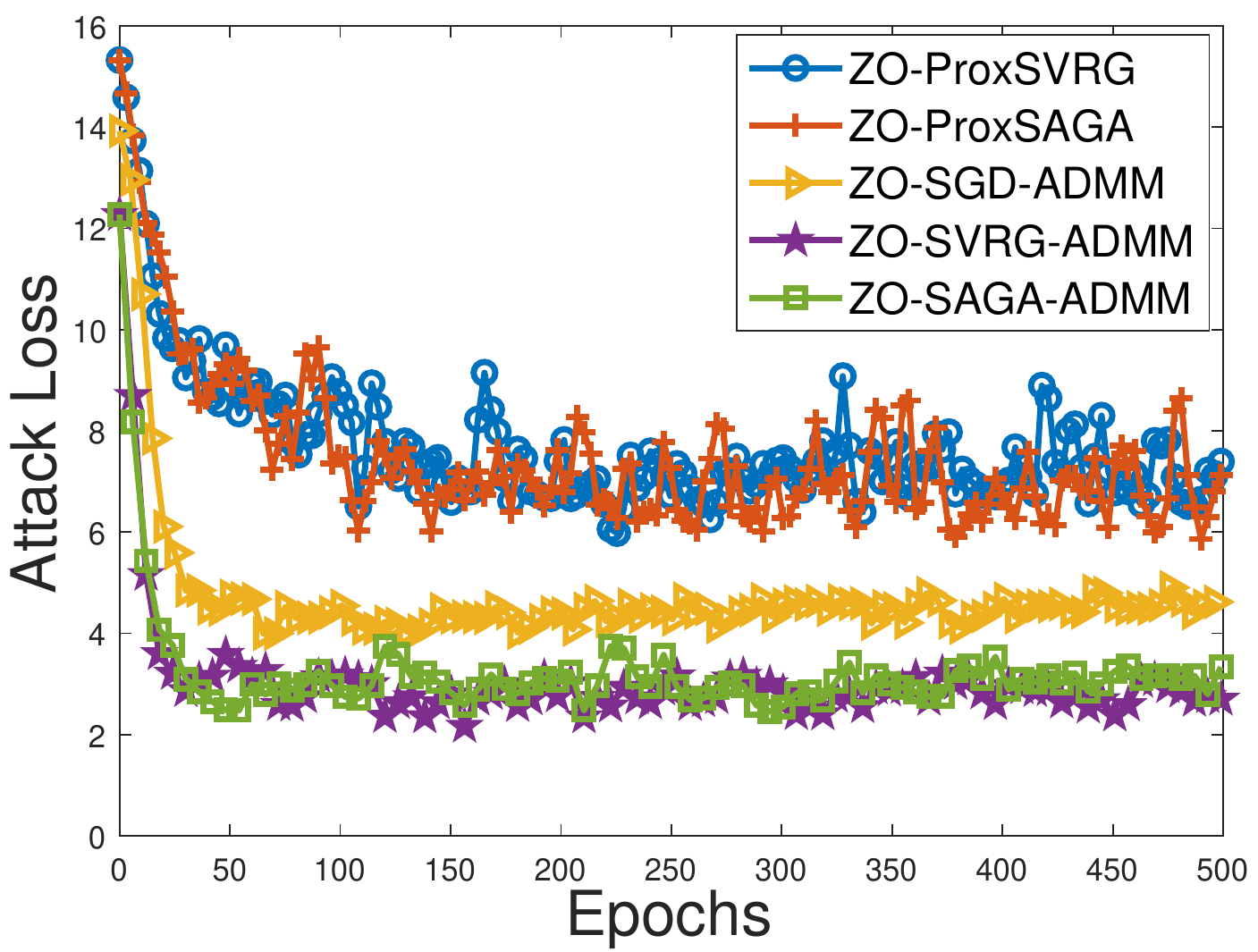}}
\subfigure[\emph{CIFAR-10}]{\includegraphics[width=0.236\textwidth]{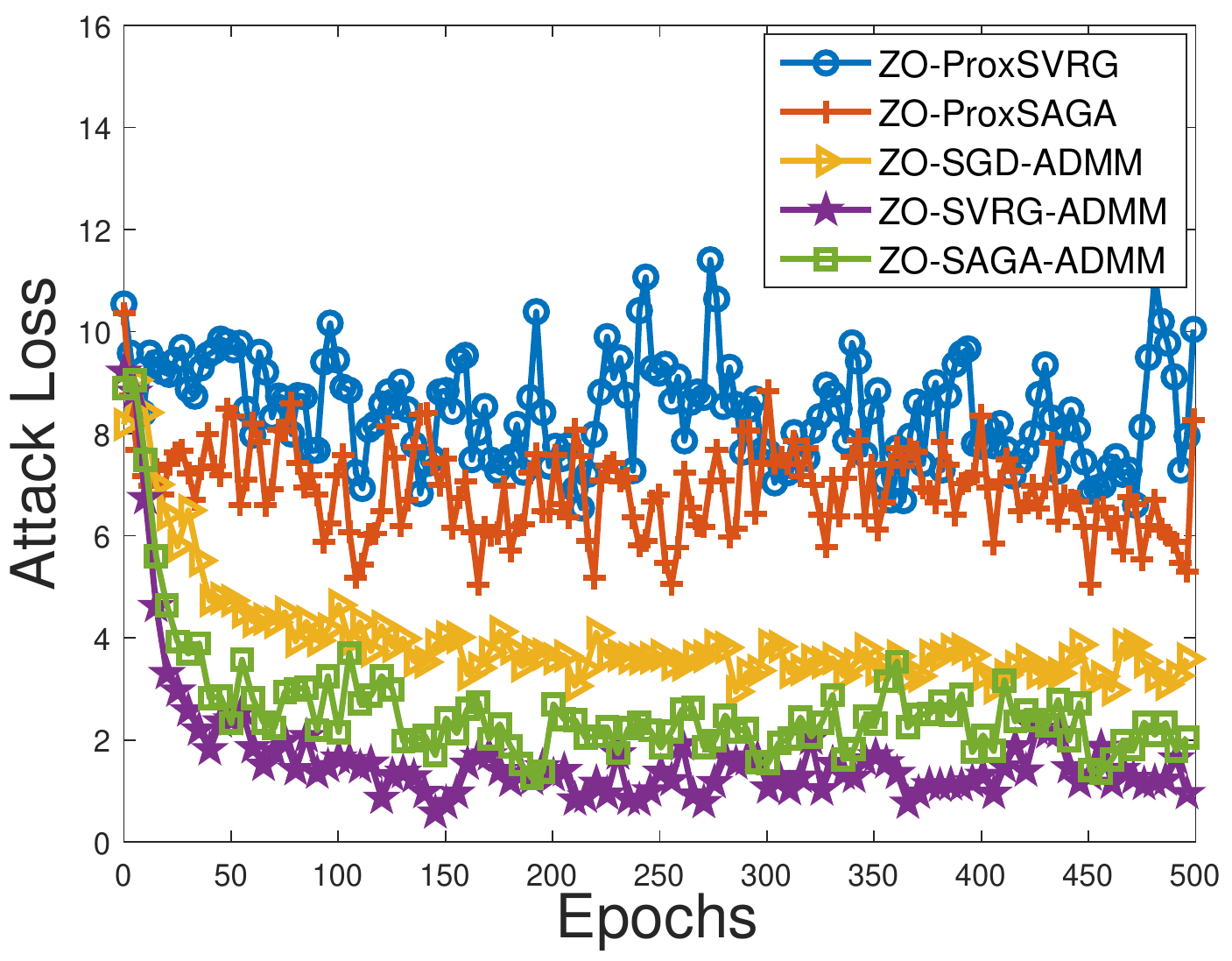}}
\caption{Attack loss on adversarial attacks black-box DNNs. }
\label{fig:2}
\end{figure}
\section{Conclusions}
In the paper, we proposed fast ZO-SVRG-ADMM and ZO-SAGA-ADMM methods based on
the coordinate smoothing gradient estimator, which only uses the objective function
values to optimize. Moreover, we prove that the proposed methods have a convergence rate of $O(\frac{1}{T})$.
In particular, our methods not only reach the existing best convergence rate $O(\frac{1}{T})$ for the nonconvex optimization,
but also are able to effectively solve many machine learning problems with the complex nonsmooth regularizations.
\section*{Acknowledgments}
F.H., S.G., H.H. were partially supported by
U.S. NSF IIS 1836945, IIS 1836938, DBI 1836866, IIS 1845666, IIS 1852606, IIS 1838627, IIS 1837956.
S.C. was partially supported by the NSFC under Grant No. 61806093 and No. 61682281, and the Key Program of NSFC under Grant No. 61732006.

\bibliographystyle{named}
\bibliography{ijcai19}

\begin{onecolumn}

\begin{appendices}

\section{ Supplementary Materials }
In this section, we study at detail the convergence properties of both the ZO-SVRG-ADMM and ZO-SAGA-ADMM algorithms.

\textbf{Notations:} To make the paper easier to follow, we give
the following notations:
\begin{itemize}
\item $[k] = \{1,2,\cdots,k\}$ and $[j:k] = \{j,j+1,\cdots,k\}$ for all $1\leq j \leq k$.
\item $\|\cdot\|$ denotes the vector $\ell_2$ norm and the matrix spectral norm, respectively.
\item $\|x\|_G=\sqrt{x^TGx}$, where $G$ is a positive definite matrix.
\item $\sigma^A_{\min}$ and $\sigma^A_{\max}$ denote the minimum and maximum eigenvalues of $A^TA$, respectively.
\item $\sigma^{B_j}_{\max}$ denotes the maximum eigenvalues of $B_j^TB_j$ for all $j\in[k]$, and $\sigma^{B}_{\max} = \max_{j=1}^k \sigma^{B_j}_{\max}$.
\item $\sigma_{\min}(G)$ and $\sigma_{\max}(G)$ denote the minimum and maximum eigenvalues of matrix $G$, respectively;
the conditional number $\kappa_G = \frac{\sigma_{\max}(G)}{\sigma_{\min}(G)}$.
\item $\sigma_{\min}(H_j)$ and $\sigma_{\max}(H_j)$ denote the minimum and maximum eigenvalues of matrix $H_j$ for all $j\in[k]$, respectively;
$\sigma_{\min}(H) = \min_{j=1}^k\sigma_{\min}(H_j)$ and $\sigma_{\max}(H) = \max_{j=1}^k\sigma_{\max}(H_j)$.
\item $\mu$ denotes the smoothing parameter of the gradient estimator.
\item $\eta$ denotes the step size of updating variable $x$.
\item $L$ denotes the Lipschitz constant of $\nabla f(x)$.
\item $b$ denotes the mini-batch size of stochastic gradient.
\item $T$, $m$ and $S$ are the total number of iterations, the number of iterations in the inner loop, and the number
                 of iterations in the outer loop, respectively.

\end{itemize}

\subsection{ Theoretical Analysis of the ZO-SVRG-ADMM}
In this subsection, we in detail give the convergence analysis of the ZO-SVRG-ADMM algorithm.
First, we give some useful lemmas as follows:

\begin{lemma} \label{lem:1}
Suppose the sequence $\big\{(x^s_t,y_{[k]}^{s,t},\lambda^s_t)_{t=1}^m\big\}_{s=1}^S$ is generated by Algorithm \ref{alg:1}, the following inequality holds
 \begin{align}
 \mathbb{E}\|\lambda^s_{t+1}-\lambda^s_{t}\|^2 \leq & \frac{18L^2d }{\sigma^A_{\min} b} \big( \|x^s_t - \tilde{x}^s\|^2
 + \|x^s_{t-1} - \tilde{x}^s\|^2\big)
  + \frac{3\sigma^2_{\max}(G)}{\sigma^A_{\min}\eta^2}\|x^s_{t+1}-x^s_t\|^2 \nonumber \\
  & + (\frac{3\sigma^2_{\max}(G)}{\sigma^A_{\min}\eta^2}+\frac{9L^2}{\sigma^A_{\min}})\|x^s_{t}-x^s_{t-1}\|^2
  + \frac{9L^2d^2\mu^2}{\sigma^A_{\min}}.
 \end{align}
\end{lemma}

\begin{proof}
 Using the optimal condition for the step 9 of Algorithm \ref{alg:1}, we have
 \begin{align}
   \hat{g}^s_t + \frac{1}{\eta}G(x^s_{t+1}-x^s_t) - A^T\lambda^s_t + \rho A^T( Ax^s_{t+1} + \sum_{j=1}^kB_jy_j^{s,t+1} -c) = 0.
 \end{align}
 By the step 11 of Algorithm \ref{alg:1}, we have
 \begin{align} \label{eq:A65}
  A^T\lambda^s_{t+1} = \hat{g}^s_t + \frac{1}{\eta}G(x^s_{t+1}-x^s_t).
 \end{align}
 It follows that
 \begin{align} \label{eq:A66}
  \lambda^s_{t+1} = (A^T)^+ \big( \hat{g}^s_t + \frac{1}{\eta}G(x^s_{t+1}-x^s_t) \big),
 \end{align}
where $(A^T)^+$ is the pseudoinverse of $A^T$. By Assumption 4, \emph{i.e.,} $A$ is a full column matrix, we have $(A^T)^+=A(A^TA)^{-1}$.
Then we have
 \begin{align} \label{eq:A67}
 \mathbb{E} \|\lambda^s_{t+1}-\lambda^s_t\|^2 & = \mathbb{E}\|(A^T)^+ \big( \hat{g}^s_t - \hat{g}^s_{t-1} + \frac{G}{\eta}(x^s_{t+1}-x^s_t) - \frac{G}{\eta}(x^s_{t}-x^s_{t-1}) \big)\|^2 \nonumber \\
 &\leq \frac{1}{\sigma^A_{\min}}\big[3 \mathbb{E}\|\hat{g}^s_t - \hat{g}^s_{t-1}\|^2 +\frac{3\sigma^2_{\max}(G)}{\eta^2}\mathbb{E} \|x^s_{t+1}-x^s_t\|^2
  + \frac{3\sigma^2_{\max}(G)}{\eta^2} \mathbb{E}\|x^s_{t}-x^s_{t-1}\|^2 \big],
 \end{align}
where $\sigma^A_{\min}$ denotes the minimum eigenvalues of $A^TA$.

 Next, considering the upper bound of $\|\hat{g}^s_t - \hat{g}^s_{t-1}\|^2$, we have
 \begin{align} \label{eq:A68}
   \mathbb{E} \|\hat{g}^s_t - \hat{g}^s_{t-1}\|^2 & =  \mathbb{E}\|\hat{g}^s_t - \nabla f(x^s_t) + \nabla f(x^s_t) -  \nabla f(x^s_{t-1}) + \nabla f(x^s_{t-1}) - \hat{g}^s_{t-1}\|^2 \nonumber \\
  & \leq 3 \mathbb{E}\|\hat{g}^s_t - \nabla f(x^s_t)\|^2 + 3 \mathbb{E}\|\nabla f(x^s_t) -  \nabla f(x^s_{t-1})\|^2 + 3 \mathbb{E}\|\nabla f(x^s_{t-1}) - \hat{g}^s_{t-1}\|^2 \nonumber \\
  & \leq \frac{6L^2d}{b}\|x^s_t - \tilde{x}^s\|^2 + \frac{3L^2d^2\mu^2}{2} + \frac{6L^2d}{b}\|x^s_{t-1} - \tilde{x}^s\|^2  + \frac{3L^2d^2\mu^2}{2} +  3\|\nabla f(x^s_t) -  \nabla f(x^s_{t-1})\|^2 \nonumber \\
  & \leq  \frac{6L^2d}{b} \big(\|x^s_t - \tilde{x}^s\|^2 +  \|x^s_{t-1} - \tilde{x}^s\|^2 \big) + 3L^2\|x^s_t - x^s_{t-1}\|^2 + 3L^2d^2\mu^2,
 \end{align}
 where the second inequality holds by Lemma 1 of \cite{Huang2019faster} and the third inequality holds by Assumption 1.
 Finally, combining \eqref{eq:A67} and \eqref{eq:A68}, we obtain the above result.
\end{proof}

\begin{lemma} \label{lem:2}
 Suppose the sequence $\{(x^{s}_t,y_{[k]}^{s,t},\lambda^{s}_t)_{t=1}^m\}_{s=1}^S$ is generated from Algorithm \ref{alg:1},
 and define a \emph{Lyapunov} function:
 \begin{align}
 R^s_t = \mathbb{E}\big[\mathcal{L}_{\rho} (x^s_t,y_{[k]}^{s,t},\lambda^s_t) + (\frac{3\sigma^2_{\max}(G)}{\sigma^A_{\min}\eta^2\rho} + \frac{9L^2}{\sigma^A_{\min}\rho})\|x^s_{t}-x^s_{t-1}\|^2
 + \frac{18 L^2d }{\sigma^A_{\min}\rho b}\|x^s_{t-1}-\tilde{x}^s\|^2 + c_t\|x^s_{t}-\tilde{x}^s\|^2\big]
 \end{align}
 where the positive sequence $\{c_t\}$ satisfies, for $s =1,2,\cdots,S$
 \begin{equation*}
  c_t= \left\{
  \begin{aligned}
  & \frac{36 L^2d }{\sigma^A_{\min}\rho b} +
     \frac{2Ld}{b} + (1+\beta)c_{t+1}, \ 1 \leq t \leq m, \\
  & \\
  & 0, \ t \geq m+1.
  \end{aligned}
  \right.\end{equation*}
It follows that
\begin{align}
\frac{1}{T}\sum_{s=1}^S \sum_{t=0}^{m-1} (\sum_{j=1}^k \|y_j^{s,t}-y_j^{s,t+1}\|^2 + \frac{d}{b}\|x^s_t-\tilde{x}^s\|^2_2 + \|x^s_{t+1}-x^s_t\|^2)
\leq \frac{R^1_0 - R^*}{\gamma T}  + \frac{9L^2d^2\mu^2}{\gamma\sigma^A_{\min}\rho} + \frac{L d^2 \mu^2}{4\gamma},
\end{align}
where $\gamma = \min(\sigma_{\min}^H, L, \chi_t)$, $\chi_t \geq \frac{3\sqrt{71}\kappa_GLd^{l}}{2\alpha} >0 \ (l=0,0.5,1)$,
and $R^*$ denotes a lower bound of $R^s_t$.
\end{lemma}

\begin{proof}
By the optimal condition of step 8 in Algorithm \ref{alg:1},
we have, for $j\in [k]$
\begin{align}
0 & =(y_j^{s,t}-y_j^{s,t+1})^T\big(\partial \psi_j(y_j^{s,t+1}) - B^T\lambda_t^s + \rho B^T(Ax_t^s + \sum_{i=1}^jB_iy_i^{s,t+1} + \sum_{i=j+1}^kB_iy_i^{s,t}-c) + H_j(y_j^{s,t+1}-y_j^{s,t})\big) \nonumber \\
& \leq \psi_j(y_j^{s,t})- \psi_j(y_j^{s,t+1}) - (\lambda_t^s)^T(B_jy_j^{s,t}-B_jy_j^{s,t+1}) + \rho(By_j^{s,t}-By_j^{s,t+1})^T(Ax_t^s + \sum_{i=1}^jB_iy_i^{s,t+1} + \sum_{i=j+1}^kB_iy_i^{s,t}-c) \nonumber \\
& \quad - \|y_j^{s,t+1}-y_j^{s,t}\|^2_{H_j} \nonumber \\
& = \psi_j(y_j^{s,t})- \psi_j(y_j^{s,t+1}) - (\lambda_t^s)^T(Ax_t^s+\sum_{i=1}^{j-1}B_iy_i^{s,t+1} + \sum_{i=j}^kB_iy_i^{s,t}-c) + (\lambda_t^s)^T(Ax_t^s+\sum_{i=1}^jB_iy_i^{s,t+1}+ \sum_{i=j+1}^kB_iy_i^{s,t}-c) \nonumber \\
& \quad  + \frac{\rho}{2}\|Ax_t^s +\sum_{i=1}^{j-1}B_iy_i^{s,t+1} + \sum_{i=j}^kB_iy_i^{s,t}-c\|^2 - \frac{\rho}{2}\|Ax_t^s+\sum_{i=1}^jB_iy_i^{s,t+1}+ \sum_{i=j+1}^kB_iy_i^{s,t}-c\|^2
  - \|y_j^{s,t+1}-y_j^{s,t}\|^2_{H_j} \nonumber \\
& \quad  -\frac{\rho}{2}\|B_jy_j^{s,t}-B_jy_j^{s,t+1}\|^2  \nonumber \\
& = \underbrace{ f(x_t^s) \!+ \! \sum_{i=1}^{j-1}\psi_i(y_i^{s,t+1})  \!+ \! \sum_{i=j}^{k}\psi_i(y_i^{s,t})  \!- \! (\lambda_t^s)^T(Ax_t^s  \!+ \! \sum_{i=1}^{j-1}B_iy_i^{s,t+1}  \!+ \! \sum_{i=j}^kB_iy_i^{s,t}-c)  \!+ \! \frac{\rho}{2}\|Ax_t^s  \!+ \!\sum_{i=1}^{j-1}B_iy_i^{s,t+1}  \!+ \! \sum_{i=j}^kB_iy_i^{s,t}-c\|^2}_{\mathcal{L}_{\rho} (x_t^s,y_{[j-1]}^{s,t+1},y_{[j:k]}^{s,t},\lambda_t^s)}  \nonumber \\
& \!-\! \big( \underbrace{f(x_t^s) \!+\! \sum_{i=1}^{j}\psi_i(y_i^{s,t+1}) \!+\! \sum_{i=j+1}^{k}\psi_i(y_i^{s,t}) \!-\! (\lambda_t^s)^T(Ax_t^s \!+\! \sum_{i=1}^jB_iy_i^{s,t+1} \!+\! \sum_{i=j+1}^kB_iy_i^{s,t}\!-\!c)
\!+\! \frac{\rho}{2}\|Ax_t^s\!+\!\sum_{i=1}^jB_iy_i^{s,t+1} \!+\! \sum_{i=j+1}^kB_iy_i^{s,t}\!-\!c\|^2}_{\mathcal{L}_{\rho} (x_t^s,y_{[j]}^{s,t+1},y_{[j+1:k]}^{s,t},\lambda_t^s)} \big) \nonumber \\
& \quad - \|y_j^{s,t+1}-y_j^{s,t}\|^2_{H_j} -\frac{\rho}{2}\|B_jy_j^{s,t}-B_jy_j^{s,t+1}\|^2  \nonumber \\
& \leq \mathcal{L}_{\rho} (x_t^s,y_{[j-1]}^{s,t+1},y_{[j:k]}^{s,t},\lambda_t^s) - \mathcal{L}_{\rho} (x_t^s,y_{[j]}^{s,t+1},y_{[j+1:k]}^{s,t},\lambda_t^s)
- \sigma_{\min}(H_j)\|y_j^{s,t}-y_j^{s,t+1}\|^2,
\end{align}
where the first inequality holds by the convexity of function $\psi_j(y)$,
and the second equality follows by applying the equality
$(a-b)^Tb = \frac{1}{2}(\|a\|^2-\|b\|^2-\|a-b\|^2)$ on the term $(By_j^{s,t}-By_j^{s,t+1})^T(Ax_t^s + \sum_{i=1}^jB_iy_i^{s,t+1} + \sum_{i=j+1}^kB_iy_i^{s,t}-c)$.
Thus, we have, for all $j\in[k]$
\begin{align} \label{eq:A10-1}
\mathcal{L}_{\rho} (x_t^s,y_{[j]}^{s,t+1},y_{[j+1:k]}^{s,t},\lambda_t^s) \leq \mathcal{L}_{\rho} (x_t^s,y_{[j-1]}^{s,t+1},y_{[j:k]}^{s,t},\lambda_t^s)
 - \sigma_{\min}(H_j)\|y_j^{s,t}-y_j^{s,t+1}\|^2.
\end{align}
Telescoping inequality \eqref{eq:A10-1} over $j$ from $1$ to $k$, we obtain
\begin{align} \label{eq:A73}
 \mathcal{L}_{\rho} (x_t^s,y^{s,t+1}_{[k]},\lambda_t^s) \leq \mathcal{L}_{\rho} (x_t^s,y^{s,t}_{[k]},\lambda_t^s)
 - \sigma_{\min}^H\sum_{j=1}^k \|y_j^{s,t}-y_j^{s,t+1}\|^2,
\end{align}
where $\sigma_{\min}^H=\min_{j\in[k]}\sigma_{\min}(H_j)$.

By Assumption 1, we have
\begin{align} \label{eq:A74}
0 \leq f(x^s_t) - f(x^s_{t+1}) + \nabla f(x^s_t)^T(x^s_{t+1}-x^s_t) + \frac{L}{2}\|x^s_{t+1}-x^s_t\|^2.
\end{align}
Using the optimal condition of the step 9 in Algorithm \ref{alg:1},
we have
\begin{align} \label{eq:A75}
 0 = (x^s_t-x^s_{t+1})^T \big( \hat{g}^s_t - A^T\lambda^s_t + \rho A^T(Ax^s_{t+1} + \sum_{j=1}^kB_jy_j^{s,t+1}-c) + \frac{G}{\eta}(x^s_{t+1}-x^s_t) \big).
\end{align}
Combining \eqref{eq:A74} and \eqref{eq:A75}, we have
\begin{align}
 0 & \leq f(x^s_t) - f(x^s_{t+1}) + \nabla f(x^s_t)^T(x^s_{t+1}-x^s_t) + \frac{L}{2}\|x^s_{t+1}-x^s_t\|^2 \nonumber \\
 & \quad + (x^s_t-x^s_{t+1})^T \big( \hat{g}^s_t - A^T\lambda^s_t + \rho A^T(Ax^s_{t+1} + \sum_{j=1}^kB_jy_j^{s,t+1}-c) + \frac{G}{\eta}(x^s_{t+1}-x^s_t) \big)  \nonumber \\
 & = f(x^s_t) - f(x^s_{t+1}) + \frac{L}{2}\|x^s_t-x^s_{t+1}\|^2 - \frac{1}{\eta}\|x^s_t - x^s_{t+1}\|^2_G + (x^s_t-x^s_{t+1})^T(\hat{g}^s_t-\nabla f(x^s_t)) \nonumber \\
 & \quad -(\lambda^s_t)^T(Ax^s_t-Ax^s_{t+1}) + \rho(Ax^s_t - Ax^s_{t+1})^T(Ax^s_{t+1} + \sum_{j=1}^kB_jy_j^{s,t+1}-c) \nonumber \\
 & \mathop{=}^{(i)} f(x^s_t) - f(x^s_{t+1}) + \frac{L}{2}\|x^s_t-x^s_{t+1}\|^2 - \frac{1}{\eta}\|x^s_t - x^s_{t+1}\|^2_G
 + (x^s_t-x^s_{t+1})^T(\hat{g}^s_t-\nabla f(x^s_t)) -(\lambda^s_t)^T(Ax^s_t + \sum_{j=1}^kB_jy_j^{s,t+1}-c)\nonumber \\
 & \quad  + (\lambda^s_t)^T(Ax^s_{t+1}+ \sum_{j=1}^kB_jy_j^{s,t+1}-c) + \frac{\rho}{2}\big(\|Ax^s_{t} + \sum_{j=1}^kB_jy_j^{s,t+1}-c\|^2
 - \|Ax^s_{t+1} + \sum_{j=1}^kB_jy_j^{s,t+1}-c\|^2 - \|Ax^s_t - Ax^s_{t+1}\|^2 \big) \nonumber \\
 & = \underbrace{f(x^s_t) + \sum_{j=1}^k \psi_j(x^s_{t+1}) -(\lambda^s_t)^T(Ax^s_t + \sum_{j=1}^kB_jy_j^{s,t+1}-c) + \frac{\rho}{2}\|Ax^s_{t} + \sum_{j=1}^kB_jy_j^{s,t+1}-c\|^2}_{\mathcal{L}_{\rho} (x^s_t,y_{[k]}^{s,t+1},\lambda^s_t)} \nonumber \\
 & \quad -\big(\underbrace{ f(x^s_{t+1}) + \sum_{j=1}^k \psi_j(x^s_{t+1}) -(\lambda^s_t)^T(Ax^s_{t+1} + \sum_{j=1}^kB_jy_j^{s,t+1}-c) + \frac{\rho}{2}\|Ax^s_{t+1} + \sum_{j=1}^kB_jy_j^{s,t+1}-c\|^2}_{\mathcal{L}_{\rho} (x^s_{t+1},y_{[k]}^{s,t+1},\lambda^s_t)} \big) \nonumber \\
 & \quad + \frac{L}{2}\|x^s_t-x^s_{t+1}\|^2 + (x^s_t-x^s_{t+1})^T(\hat{g}^s_t-\nabla f(x^s_t))  -\frac{1}{\eta}\|x^s_t - x^s_{t+1}\|^2_G - \frac{\rho}{2}\|Ax^s_t - Ax^s_{t+1}\|^2 \nonumber \\
 & \leq \mathcal{L}_{\rho} (x^s_t,y_{[k]}^{s,t+1},\lambda^s_t) -  \mathcal{L}_{\rho} (x^s_{t+1},y_{[k]}^{s,t+1},\lambda^s_t)
 - (\frac{\sigma_{\min}(G)}{\eta}+\frac{\rho \sigma^A_{\min}}{2}-\frac{L}{2})\|x^s_t - x^s_{t+1}\|^2 +(x^s_t-x^s_{t+1})^T(\hat{g}^s_t-\nabla f(x^s_t)) \nonumber \\
 & \mathop{\leq}^{(ii)}  \mathcal{L}_{\rho} (x^s_t,y_{[k]}^{s,t+1},\lambda^s_t) -  \mathcal{L}_{\rho} (x^s_{t+1},y_{[k]}^{s,t+1},\lambda^s_t)
 - (\frac{\sigma_{\min}(G)}{\eta}+\frac{\rho \sigma^A_{\min}}{2}-L)\|x^s_t - x^s_{t+1}\|^2 + \frac{1}{2L}\|\hat{g}^s_t-\nabla f(x^s_t)\|^2 \nonumber \\
 & \mathop{\leq}^{(iii)} \mathcal{L}_{\rho} (x^s_t,y_{[k]}^{s,t+1},\lambda^s_t) -  \mathcal{L}_{\rho} (x^s_{t+1},y_{[k]}^{s,t+1},\lambda^s_t)
 - (\frac{\sigma_{\min}(G)}{\eta}+\frac{\rho \sigma^A_{\min}}{2}-L)\|x^s_t - x^s_{t+1}\|^2 +\frac{ L d}{b}\|x^s_t-\tilde{x}^s\|^2 + \frac{L d^2 \mu^2}{4},
\end{align}
where the equality $(i)$ holds by applying the equality
$(a-b)^Tb = \frac{1}{2}(\|a\|^2-\|b\|^2-\|a-b\|^2)$ on the term $(Ax^s_t - Ax^s_{t+1})^T(Ax^s_{t+1}+\sum_{j=1}^kB_jy_j^{s,t+1}-c)$, the inequality
$(ii)$ holds by the inequality $a^Tb \leq \frac{L}{2}\|a\|^2 + \frac{1}{2L}\|b\|^2$,
and the inequality $(iii)$ holds by Lemma 1 of \cite{Huang2019faster}. Thus, we obtain
\begin{align} \label{eq:A77}
\mathcal{L}_{\rho} (x^s_{t+1},y_{[k]}^{s,t+1},\lambda^s_t) \leq & \mathcal{L}_{\rho} (x^s_t,y_{[k]}^{s,t+1},\lambda^s_t) -
(\frac{\sigma_{\min}(G)}{\eta}+\frac{\rho \sigma^A_{\min}}{2}-L)\|x^s_t - x^s_{t+1}\|^2 +\frac{ L d}{b}\|x^s_t-\tilde{x}^s\|^2 + \frac{L d^2 \mu^2}{4}.
\end{align}

Using the step 10 in Algorithm \ref{alg:1}, we have
\begin{align} \label{eq:A78}
\mathcal{L}_{\rho} (x^s_{t+1},y_{[k]}^{s,t+1},\lambda^s_{t+1}) -
\mathcal{L}_{\rho} (x^s_{t+1},y_{[k]}^{s,t+1},\lambda^s_t)
& = \frac{1}{\rho}\|\lambda^s_{t+1}-\lambda^s_t\|^2 \nonumber \\
& \leq  \frac{18 L^2d  }{\sigma^A_{\min} b\rho} \big( \|x^s_t - \tilde{x}^s\|^2 + \|x^s_{t-1} - \tilde{x}^s\|^2\big)
 + \frac{3\sigma^2_{\max}(G)}{\sigma^A_{\min}\eta^2\rho}\|x^s_{t+1}-x^s_t\|^2 \nonumber \\
& \quad + (\frac{3\sigma^2_{\max}(G)}{\sigma^A_{\min}\eta^2\rho}+\frac{9L^2}{\sigma^A_{\min}\rho})\|x^s_{t}-x^s_{t-1}\|^2
  + \frac{9L^2d^2\mu^2}{\sigma^A_{\min}\rho}.
\end{align}
Combining \eqref{eq:A73}, \eqref{eq:A77} and \eqref{eq:A78}, we have
\begin{align}
\mathcal{L}_{\rho} (x^s_{t+1},y_{[k]}^{s,t+1},\lambda^s_{t+1}) & \leq \mathcal{L}_{\rho} (x^s_t,y_{[k]}^{s,t},\lambda^s_t)
- \sigma_{\min}^H\sum_{j=1}^k \|y_j^{s,t}-y_j^{s,t+1}\|^2 - (\frac{\sigma_{\min}(G)}{\eta}+\frac{\rho \sigma^A_{\min}}{2}-L)\|x^s_t - x^s_{t+1}\|^2  \nonumber \\
& \quad +\frac{ L d}{b}\|x^s_t-\tilde{x}^s\|^2 + \frac{18 L^2d }{\sigma^A_{\min} b\rho}
\big( \|x^s_t - \tilde{x}^s\|^2 + \|x^s_{t-1} - \tilde{x}^s\|^2\big)
+ \frac{3\sigma^2_{\max}(G)}{\sigma^A_{\min}\eta^2\rho}\|x^s_{t+1}-x^s_t\|^2 \nonumber \\
& \quad + (\frac{3\sigma^2_{\max}(G)}{\sigma^A_{\min}\eta^2\rho}+\frac{9L^2}{\sigma^A_{\min}\rho})\|x^s_{t}-x^s_{t-1}\|^2
+ \frac{9L^2d^2\mu^2}{\sigma^A_{\min}\rho} + \frac{L d^2 \mu^2}{4}.
\end{align}

Next, we define a \emph{Lyapunov} function $R^s_t$ as follows:
\begin{align} \label{eq:79}
 R^s_t = \mathbb{E}\big[\mathcal{L}_{\rho} (x^s_t,y_{[k]}^{s,t},\lambda^s_t) + (\frac{3\sigma^2_{\max}(G)}{\sigma^A_{\min}\eta^2\rho} + \frac{9L^2}{\sigma^A_{\min}\rho})\|x^s_{t}-x^s_{t-1}\|^2
 + \frac{18 L^2d }{\sigma^A_{\min}\rho b}\|x^s_{t-1}-\tilde{x}^s\|^2 + c_t\|x^s_{t}-\tilde{x}^s\|^2\big].
\end{align}
Considering the upper bound of $\|x^s_{t+1}-\tilde{x}^s\|^2$, we have
\begin{align}  \label{eq:80}
\|x^s_{t+1}-x^s_t + x^s_t - \tilde{x}^s\|^2
& = \|x^s_{t+1}-x^s_t\|^2 + 2(x^s_{t+1}-x^s_t)^T(x^s_t-\tilde{x}^s) +\|x^s_t -\tilde{x}^s\|^2 \nonumber \\
& \leq \|x^s_{t+1}-x^s_t\|^2 + 2\big(\frac{1}{2\beta} \|x^s_{t+1}-x^s_t\|^2 + \frac{\beta}{2}\|x^s_t-\tilde{x}^s\|^2\big)
+ \|x^s_t -\tilde{x}^s\|^2 \nonumber \\
& = (1+1/\beta)\|x^s_{t+1}-x^s_t\|^2 +(1+\beta)\|x^s_t -\tilde{x}^s\|^2,
\end{align}
where the above inequality holds by the
Cauchy-Schwarz inequality with $\beta>0$.
Combining \eqref{eq:79} with \eqref{eq:80}, then we obtain
\begin{align} \label{eq:A81}
R^s_{t+1} & = \mathbb{E}\big[\mathcal{L}_{\rho}(x^s_{t+1},y_{[k]}^{s,t+1},\lambda^s_{t+1}) + (\frac{3\sigma^2_{\max}(G)}{\sigma^A_{\min}\eta^2\rho}+\frac{9L^2}{\sigma^A_{\min}\rho}) \|x^s_{t+1}-x^s_{t}\|^2 + \frac{18 L^2d }{\sigma^A_{\min} b\rho} \|x^s_{t}-\tilde{x}^s\|^2 + c_{t+1}\|x^s_{t+1}-\tilde{x}^s\|^2\big] \nonumber \\
& \leq \mathcal{L}_{\rho} (x^s_t,y_{[k]}^{s,t},\lambda^s_t) + (\frac{3\sigma^2_{\max}(G)}{\sigma^A_{\min}\eta^2\rho}+\frac{9L^2}{\sigma^A_{\min}\rho})\|x^s_{t}-x^s_{t-1}\|^2
+ \frac{18 L^2d }{\sigma^A_{\min} b\rho} \|x^s_{t-1}-\tilde{x}^s\|^2 \nonumber \\
&\quad + \big(\frac{36 L^2d }{\sigma^A_{\min} b\rho} + \frac{2L d}{b}  +(1+\beta)c_{t+1}\big)\|x^s_t-\tilde{x}^s\|^2 - \sigma_{\min}^H\sum_{j=1}^k \|y_j^{s,t}-y_j^{s,t+1}\|^2 \nonumber\\
& \quad - \big( \frac{\sigma_{\min}(G)}{\eta} + \frac{\rho\sigma^A_{\min}}{2} - L - \frac{6\sigma^2_{\max}(G)}{\sigma^A_{\min}\eta^2\rho}-\frac{9L^2}{\sigma^A_{\min}\rho}
- (1+1/\beta)c_{t+1} \big)\|x^s_t - x^s_{t+1}\|^2 \nonumber \\
& \quad- \frac{L d }{ b}\|x^s_{t}-\tilde{x}^s\|^2 + \frac{9L^2d^2\mu^2}{\sigma^A_{\min}\rho} + \frac{L d^2 \mu^2}{4}\nonumber \\
& \leq R^s_t - \sigma_{\min}^H\sum_{j=1}^k \|y_j^{s,t}-y_j^{s,t+1}\|^2 - \frac{L d}{b}\|x^s_{t}-\tilde{x}^s\|^2 - \chi_t \|x^s_t - x^s_{t+1}\|^2  + \frac{9L^2d^2\mu^2}{\sigma^A_{\min}\rho} + \frac{L d^2 \mu^2}{4},
\end{align}
where
$c_t = \frac{36 L^2d }{\sigma^A_{\min} b\rho} + \frac{2L d}{b}  + (1+\beta)c_{t+1}$ and $\chi_t = \frac{\sigma_{\min}(G)}{\eta}+\frac{\rho
\sigma^A_{\min}}{2} - L - \frac{6\sigma^2_{\max}(G)}{\sigma^A_{\min}\eta^2\rho}-\frac{9L^2}{\sigma^A_{\min}\rho} - (1+1/\beta)c_{t+1}$.

Next, we will prove the relationship between $R^{s+1}_1$ and
$R^s_m$. Due to $x^{s+1}_0 = x^s_m = \tilde{x}^{s+1}$, we have
\begin{align}
 \hat{g}^{s+1}_0 = \hat{\nabla} f_{\mathcal{I}}(x^{s+1}_0) - \hat{\nabla}
 f_{\mathcal{I}}(x^{s+1}_0) + \hat{\nabla} f(x^{s+1}_0) = \hat{\nabla}
 f(x^{s+1}_0) = \hat{\nabla} f(x^{s}_m).
\end{align}
It follows that
\begin{align}
 \mathbb{E} \|\hat{g}^{s+1}_0 - \hat{g}^{s}_m\|^2 & = \mathbb{E}\|\hat{\nabla} f(x^{s}_m) - \hat{\nabla} f_{\mathcal{I}}(x^{s}_m)
 + \hat{\nabla} f_{\mathcal{I}}(\tilde{x}^{s}) - \hat{\nabla} f(\tilde{x}^{s})\|^2 \nonumber \\
 & = \| \hat{\nabla} f_{\mathcal{I}}(x^{s}_m)
 - \hat{\nabla} f_{\mathcal{I}}(\tilde{x}^{s})-\mathbb{E}_{\mathcal{I}} [\hat{\nabla} f_{\mathcal{I}}(x^{s}_m)
 - \hat{\nabla} f_{\mathcal{I}}(\tilde{x}^{s})]\|^2 \nonumber \\
 & \leq \| \hat{\nabla} f_{\mathcal{I}}(x^{s}_m) - \hat{\nabla} f_{\mathcal{I}}(\tilde{x}^{s}) \|^2 \nonumber \\
 & \leq \frac{1}{b} \sum_{i\in \mathcal{I}} \| \hat{\nabla} f_{i}(x^{s}_m) - \hat{\nabla} f_{i}(\tilde{x}^{s})\|^2 \nonumber \\
 & = \frac{1}{b} \sum_{i\in \mathcal{I}} \| \hat{\nabla} f_{i}(x^{s}_m) - \hat{\nabla} f_{i}(\tilde{x}^{s})\|^2 \nonumber \\
 & \leq \frac{L^2d}{b}\|x^s_m-\tilde{x}^s\|^2,
\end{align}
where the first inequality holds by the inequality $\mathbb{E}\|\zeta - \mathbb{E}\zeta\|^2 = \mathbb{E}\|\zeta\|^2 - \|\mathbb{E}\zeta\|^2$;
the third inequality holds by the definition of zeroth-order gradient \eqref{eq:5}.

By Lemma \ref{lem:1}, we have
\begin{align} \label{eq:A85}
 \|\lambda^{s+1}_1-\lambda^s_m\|^2 & \leq
 \frac{1}{\sigma^A_{\min}}\|\hat{g}^{s+1}_0 - \hat{g}^{s}_m + \frac{G}{\eta}(x^{s+1}_1-x^{s+1}_0) + \frac{G}{\eta}(x^s_m - x^s_{m-1})
 \|^2 \nonumber \\
 &= \frac{1}{\sigma^A_{\min}}\|\hat{\nabla} f(x^{s}_m) - \hat{g}^{s}_m + \frac{G}{\eta}(x^{s+1}_1-x^{s}_m) + \frac{G}{\eta}(x^s_m - x^s_{m-1})
 \|^2 \nonumber \\
 & \leq \frac{1}{\sigma^A_{\min}} \big( 3\|\hat{\nabla} f(x^{s}_m) - \hat{g}^{s}_m\|^2 +
 \frac{3\sigma^2_{\max}(G)}{\eta^2}\|x^{s+1}_1 - x^s_m\|^2
 + \frac{3\sigma^2_{\max}(G)}{\eta^2}\|x^s_m-x^{s}_{m-1}\|^2 \big)
 \nonumber \\
 & \leq \frac{1}{\sigma^A_{\min}} \big( 3\|\hat{\nabla} f(x^{s}_m) - \hat{g}^{s}_m\|^2 +
 \frac{3\sigma^2_{\max}(G)}{\eta^2}\|x^{s+1}_1 - x^s_m\|^2
 + \frac{3\sigma^2_{\max}(G)}{\eta^2}\|x^s_m-x^{s}_{m-1}\|^2 \big)
 \nonumber \\
 & \leq \frac{1}{\sigma^A_{\min}} \big( \frac{3L^2 d}{b}\|x^s_m-\tilde{x}^s\|^2_2
  + \frac{3\sigma^2_{\max}(G)}{\eta^2}\|x^{s+1}_1 - x^s_m\|^2
 + \frac{3\sigma^2_{\max}(G)}{\eta^2}\|x^s_m-x^{s}_{m-1}\|^2 \big).
\end{align}

Since $x^{s}_m=x^{s+1}_0$, $y_j^{s,m} = y_j^{s+1,0}$ for all $j\in [k]$ and $\lambda^{s}_m=\lambda^{s+1}_0$,
by \eqref{eq:A73}, we have
\begin{align} \label{eq:A86}
\mathcal{L}_{\rho} (x^{s+1}_0,y_{[k]}^{s+1,1},\lambda^{s+1}_0) &\leq \mathcal{L}_{\rho} (x^{s+1}_0,y_{[k]}^{s+1,0},\lambda^{s+1}_0)
 - \sigma_{\min}^H\sum_{j=1}^k \|y_j^{s+1,0}-y_j^{s+1,1}\|^2 \nonumber \\
& = \mathcal{L}_{\rho} (x^{s}_m,y_{[k]}^{s,m},\lambda^{s}_m) - \sigma_{\min}^H\sum_{j=1}^k \|y_j^{s,m}-y_j^{s+1,1}\|^2.
\end{align}
By \eqref{eq:A77}, we have
\begin{align} \label{eq:A87}
\mathcal{L}_{\rho} (x^{s+1}_{1},y_{[k]}^{s+1,1},\lambda^{s+1}_0) \leq
\mathcal{L}_{\rho} (x^{s+1}_0,y_{[k]}^{s+1,1},\lambda^{s+1}_0) -
(\frac{\sigma_{\min}(G)}{\eta}+\frac{\rho \sigma^A_{\min}}{2}-L)\|x^{s+1}_0 - x^{s+1}_{1}\|^2 + \frac{L d^2 \mu^2}{4}.
\end{align}
By \eqref{eq:A78}, we have
\begin{align} \label{eq:A88}
\mathcal{L}_{\rho} (x^{s+1}_{1},y_{[k]}^{s+1,1},\lambda^{s+1}_1) & \leq
\mathcal{L}_{\rho} (x^{s+1}_{1},y_{[k]}^{s+1,1},\lambda^{s+1}_0) + \frac{1}{\rho}\|\lambda^{s+1}_1-\lambda^{s+1}_0\|^2 \nonumber \\
& \leq \mathcal{L}_{\rho} (x^{s+1}_{1},y_{[k]}^{s+1,1},\lambda^{s+1}_0) + \frac{1}{\sigma^A_{\min}\rho} \big( \frac{3 L^2d}{b}\|x^s_m-\tilde{x}^s\|^2_2 \nonumber \\
& \quad + \frac{3\sigma^2_{\max}(G)}{\eta^2}\|x^{s+1}_1 - x^s_m\|^2 + \frac{3\sigma^2_{\max}(G)}{\eta^2}\|x^s_m-x^{s}_{m-1}\|^2 \big).
\end{align}
where the second inequality holds by \eqref{eq:A85}.

Combining \eqref{eq:A86}, \eqref{eq:A87} with \eqref{eq:A88}, we have
\begin{align}
\mathcal{L}_{\rho} (x^{s+1}_{1},y_{[k]}^{s+1,1},\lambda^{s+1}_1) & \leq
\mathcal{L}_{\rho} (x^{s}_m,y_{[k]}^{s,m},\lambda^{s}_m) - \sigma_{\min}^H\sum_{j=1}^k \|y_j^{s,m}-y_j^{s+1,1}\|^2 -(\frac{\sigma_{\min}(G)}{\eta}+\frac{\rho \sigma^A_{\min}}{2}-L)
\|x^{s+1}_0 - x^{s+1}_{1}\|^2 \nonumber \\
& \quad  + \frac{1}{\sigma^A_{\min}\rho} \big( \frac{3L^2d}{b}\|x^s_m-\tilde{x}^s\|^2_2+\frac{3\sigma^2_{\max}(G)}{\eta^2}\|x^{s+1}_1 - x^s_m\|^2
+ \frac{3\sigma^2_{\max}(G)}{\eta^2}\|x^s_m-x^{s}_{m-1}\|^2 \big) + \frac{L d^2\mu^2}{4}.
\end{align}

Therefore, we have
\begin{align} \label{eq:A89}
R^{s+1}_1
& = \mathbb{E}\big[\mathcal{L}_{\rho}(x^{s+1}_{1},y_{[k]}^{s+1,1},\lambda^{s+1}_1)
+ (\frac{3\sigma^2_{\max}(G)}{\sigma^A_{\min}\eta^2\rho} + \frac{9L^2}{\sigma^A_{\min}\rho})\|x^{s+1}_{1}-x^{s+1}_{0}\|^2
+ \frac{18 L^2d }{\sigma^A_{\min} b\rho}\|x^{s+1}_{0}-\tilde{x}^{s+1}\|^2 + c_1\|x^{s+1}_{1}-\tilde{x}^{s+1}\|^2\big] \nonumber \\
& = \mathcal{L}_{\rho} (x^{s+1}_{1},y_{[k]}^{s+1,1},\lambda^{s+1}_1) +
\big(\frac{3\sigma^2_{\max}(G)}{\sigma^A_{\min}\eta^2\rho} + \frac{9L^2}{\sigma^A_{\min}\rho} + c_1 \big) \|x^{s+1}_{1}-x^{s+1}_{0}\|^2 \nonumber \\
& \leq \mathcal{L}_{\rho} (x^{s}_m,y_{[k]}^{s,m},\lambda^{s}_m) +
(\frac{3\sigma^2_{\max}(G)}{\sigma^A_{\min}\eta^2\rho} + \frac{9L^2}{\sigma^A_{\min}\rho})\|x^{s}_{m}-x^{s}_{m-1}\|^2
+ \frac{18 L^2d}{\sigma^A_{\min}\rho b}\|x^s_{m-1}-\tilde{x}^s\|^2_2 + (\frac{36 L^2d }{\sigma^A_{\min}\rho b} + \frac{2Ld}{b}) \|x^s_m-\tilde{x}^s\|^2_2  \nonumber \\
& \quad - \sigma_{\min}^H\sum_{j=1}^k \|y_j^{s,m}-y_j^{s+1,1}\|^2 - \big(\frac{\sigma_{\min}(G)}{\eta}+ \frac{\rho \sigma^A_{\min}}{2} - L -
\frac{6\sigma^2_{\max}(G)}{\sigma^A_{\min}\eta^2\rho} - \frac{9L^2}{\sigma^A_{\min}\rho}- c_1 \big)\|x^{s+1}_1-x^s_m\|^2_2 \nonumber \\
& \quad - \frac{9L^2}{\sigma^A_{\min}\rho} \|x^s_m - x^s_{m-1}\|^2_2 - \frac{18 L^2d}{\sigma^A_{\min}\rho b}
\|x^s_{m-1}-\tilde{x}^s\|^2_2 - (\frac{33 L^2d }{\sigma^A_{\min}\rho b} + \frac{2Ld}{b})\|x^s_m-\tilde{x}^s\|^2_2 + \frac{L d^2\mu^2}{4} \nonumber \\
& \leq R^s_m - \sigma_{\min}^H\sum_{j=1}^k \|y_j^{s,m}-y_j^{s+1,1}\|^2 - \frac{L d}{b}\|x^s_m-\tilde{x}^s\|^2_2 - \chi_m \|x^{s+1}_1-x^s_m\|^2 + \frac{L d^2\mu^2}{4} ,
\end{align}
where $c_m = \frac{36 L^2d }{\sigma^A_{\min}\rho b} + \frac{2Ld}{b}$, and $\chi_m = \frac{\sigma_{\min}(G)}{\eta}+ \frac{\rho \sigma^A_{\min}}{2} - L -
\frac{6\sigma^2_{\max}(G)}{\sigma^A_{\min}\eta^2\rho} - \frac{9L^2}{\sigma^A_{\min}\rho}- c_1$.

Let $c_{m+1} = 0$ and $\beta=\frac{1}{m}$, recursing on $t$, we have
\begin{align}
 c_{t+1} = (\frac{36 L^2d}{\sigma^A_{\min}\rho b} + \frac{2Ld}{ b})\frac{(1+\beta)^{m-t}-1}{\beta}
 & = \frac{md}{b}(\frac{36 L^2}{\sigma^A_{\min}\rho} + 2L ) \big((1+\frac{1}{m})^{m-t}-1\big) \nonumber \\
 & \leq \frac{md}{b}(\frac{36 L^2}{\sigma^A_{\min}\rho} + 2L )(e-1) \leq  \frac{2md}{b}(\frac{36 L^2}{\sigma^A_{\min}\rho} + 2L ),
\end{align}
where the above inequality holds by $(1+\frac{1}{m})^m$ is an increasing function and $\lim_{m\rightarrow \infty}(1+\frac{1}{m})^m=e$.
It follows that, for $t=1,2,\cdots,m$
\begin{align}
 \chi_t & \geq \frac{\sigma_{\min}(G)}{\eta}+\frac{\rho
\sigma^A_{\min}}{2} - L -\frac{6\sigma^2_{\max}(G)}{\sigma^A_{\min}\eta^2\rho}-\frac{9L^2}{\sigma^A_{\min}\rho} - (1+1/\beta)\frac{2md}{b}(\frac{36 L^2}{\sigma^A_{\min}\rho} + 2L ) \nonumber \\
& = \frac{\sigma_{\min}(G)}{\eta}+\frac{\rho
\sigma^A_{\min}}{2} - L -\frac{6\sigma^2_{\max}(G)}{\sigma^A_{\min}\eta^2\rho}-\frac{9L^2}{\sigma^A_{\min}\rho} - (1+m)\frac{2md}{b}(\frac{36 L^2}{\sigma^A_{\min}\rho} + 2L ) \nonumber \\
& \geq \frac{\sigma_{\min}(G)}{\eta}+\frac{\rho
\sigma^A_{\min}}{2} - L -\frac{6\sigma^2_{\max}(G)}{\sigma^A_{\min}\eta^2\rho}-\frac{9L^2}{\sigma^A_{\min}\rho} - \frac{4m^2d}{b}(\frac{36 L^2}{\sigma^A_{\min}\rho} + 2L ) \nonumber \\
& = \underbrace{\frac{\sigma_{\min}(G)}{\eta} - L - \frac{8m^2dL}{b}}_{T_1} + \underbrace{\frac{\rho\sigma^A_{\min}}{2} - \frac{6\sigma^2_{\max}(G)}{\sigma^A_{\min}\eta^2\rho}-\frac{9L^2}{\sigma^A_{\min}\rho}
- \frac{144m^2dL^2}{b\sigma^A_{\min}\rho}}_{T_2}.
\end{align}

When $1\leq d < n^{\frac{1}{3}}$, let $m=n^{\frac{1}{3}}$, $b=dn^{\frac{2}{3}}$ (i.e., $b=d^{1-l}n^{\frac{2}{3}} \ l=0$) and $0< \eta \leq \frac{\sigma_{\min}(G)}{9L}$, we have $T_1\geq 0$.
Further, let $\eta = \frac{\alpha\sigma_{\min}(G)}{9L}\ (0<\alpha \leq 1)$
and $\rho = \frac{6\sqrt{71}\kappa_G L}{\sigma^A_{\min}\alpha}$, we have
\begin{align}
T_2 &= \frac{\rho\sigma^A_{\min}}{2} -\frac{486\kappa_G^2L^2}{\sigma^A_{\min}\rho\alpha^2}-\frac{9L^2}{\sigma^A_{\min}\rho} - \frac{144L^2}{\sigma^A_{\min}\rho} \nonumber \\
& \geq \frac{\rho\sigma^A_{\min}}{2} -\frac{639\kappa_G^2L^2}{\sigma^A_{\min}\rho\alpha^2} \nonumber \\
& = \frac{\rho\sigma^A_{\min}}{4} + \underbrace{\frac{\rho\sigma^A_{\min}}{4} - \frac{639\kappa_G^2L^2}{\sigma^A_{\min}\rho\alpha^2}}_{\geq 0} \nonumber \\
& \geq \frac{3\sqrt{71}\kappa_GL}{2\alpha},
\end{align}
where the  second inequality follows $\rho = \frac{6\sqrt{71}\kappa_GL}{\sigma^A_{\min}\alpha}$. Thus, we have
$\chi_t \geq \frac{3\sqrt{71}\kappa_GL}{2\alpha} >0$ for all $t\in \{1,2,\cdots,m\}$.

When $n^{\frac{1}{3}} \leq d < n^{\frac{2}{3}}$, let $m= n^{\frac{1}{3}}$, $b= d^{\frac{1}{2}} n^{\frac{2}{3}}$ (i.e., $b=d^{1-l}n^{\frac{2}{3}} \ l=0.5$) and $0< \eta \leq \frac{\sigma_{\min}(G)}{9\sqrt{d}L}$,
we have $T_1 \geq 0$.
Further, let $\eta = \frac{\alpha\sigma_{\min}(G)}{9\sqrt{d}L} \ (0 < \alpha \leq 1)$
and $\rho = \frac{6\sqrt{71d}\kappa_G L}{\sigma^A_{\min}\alpha}$, we have
\begin{align}
T_2  &= \frac{\rho\sigma^A_{\min}}{2} -\frac{486d\kappa_G^2L^2}{\sigma^A_{\min}\rho\alpha^2}-\frac{9L^2}{\sigma^A_{\min}\rho} - \frac{144L^2}{\sigma^A_{\min}\rho} \nonumber \\
&\geq \frac{\rho\sigma^A_{\min}}{2} - \frac{639d\kappa_G^2L^2}{\sigma^A_{\min}\rho\alpha^2} \nonumber \\
& = \frac{\rho\sigma^A_{\min}}{4} + \underbrace{\frac{\rho\sigma^A_{\min}}{4} - \frac{639d\kappa_G^2L^2}{\sigma^A_{\min}\rho\alpha^2}}_{\geq 0} \nonumber \\
& \geq \frac{3\sqrt{71d}\kappa_GL}{2\alpha},
\end{align}
where the second equality follows by $\rho = \frac{6\sqrt{71d}\kappa_G L}{\sigma^A_{\min}\alpha}$. Thus, we have
$\chi_t \geq \frac{3\sqrt{71d}\kappa_GL}{2\alpha} >0$.

When $n^{\frac{2}{3}} \leq d $, let $m=n^{\frac{1}{3}}$, $b=n^{\frac{2}{3}}$ (i.e., $b=d^{1-l}n^{\frac{2}{3}} \ l=1$) and $0< \eta \leq \frac{\sigma_{\min}(G)}{9dL}$, we have $T_1 \geq 0$.
Further, let $\eta = \frac{\alpha\sigma_{\min}(G)}{9dL} \ (0 < \alpha \leq 1)$
and $\rho = \frac{6\sqrt{71}\kappa_G d L}{\sigma^A_{\min}\alpha}$, we have
\begin{align}
T_2  &= \frac{\rho\sigma^A_{\min}}{2} -\frac{486d^2\kappa_G^2L^2}{\sigma^A_{\min}\rho\alpha^2}-\frac{9L^2}{\sigma^A_{\min}\rho} - \frac{144L^2}{\sigma^A_{\min}\rho} \nonumber \\
&\geq \frac{\rho\sigma^A_{\min}}{2} - \frac{639d^2\kappa_G^2L^2}{\sigma^A_{\min}\rho\alpha^2} \nonumber \\
& = \frac{\rho\sigma^A_{\min}}{4} + \underbrace{\frac{\rho\sigma^A_{\min}}{4} - \frac{639d^2\kappa_G^2L^2}{\sigma^A_{\min}\rho\alpha^2}}_{\geq 0} \nonumber \\
& \geq \frac{3\sqrt{71}\kappa_G d L}{2\alpha},
\end{align}
where the second equality follows by $\rho = \frac{6\sqrt{71}\kappa_G d L}{\sigma^A_{\min}\alpha}$. Thus, we have
$\chi_t \geq \frac{3\sqrt{71}\kappa_G d L}{2\alpha} >0$.

By Assumption 4. i.e., $A$ is a full column rank matrix,
we have $(A^T)^+ = A(A^T A)^{-1}$.
It follows that $\sigma_{\max}((A^T)^+)^T(A^T)^+) = \sigma_{\max}((A^TA)^{-1}) = \frac{1}{\sigma_{\min}^A}$.
Since , we have
\begin{align}
&\mathcal{L}_{\rho} (x^s_{t+1},y_{[k]}^{s,t+1},\lambda^s_{t+1})
= f(x^s_{t+1}) + \sum_{j=1}^k\psi_j(y_j^{s,t+1}) - \lambda_{t+1}^T(Ax^s_{t+1} + \sum_{j=1}^kB_jy_j^{s,t+1} - c) + \frac{\rho}{2}\|Ax^s_{t+1} + \sum_{j=1}^kB_jy_j^{s,t+1} -c\|^2 \nonumber \\
& = f(x^s_{t+1}) + \sum_{j=1}^k\psi_j(y_j^{s,t+1}) -  \langle(A^T)^+(\hat{g}^s_{t} + \frac{G}{\eta}(x^s_{t+1}-x^s_t)), Ax^s_{t+1} + \sum_{j=1}^kB_jy_j^{s,t+1} -c\rangle + \frac{\rho}{2}\|Ax^s_{t+1} + \sum_{j=1}^kB_jy_j^{s,t+1}-c\|^2 \nonumber \\
& = f(x^s_{t+1}) + \sum_{j=1}^k\psi_j(y_j^{s,t+1}) - \langle(A^T)^+(\hat{g}^s_{t} - \nabla f(x^s_{t}) + \nabla f(x^s_{t})+ \frac{G}{\eta}(x^s_{t+1}-x^s_t)), Ax^s_{t+1} + \sum_{j=1}^kB_jy_j^{s,t+1} -c\rangle  \nonumber \\
& \quad +  \frac{\rho}{2}\|Ax^s_{t+1} + \sum_{j=1}^kB_jy_j^{s,t+1} -c\|^2 \nonumber \\
& \geq f(x^s_{t+1}) + \sum_{j=1}^k\psi_j(y_j^{s,t+1}) - \frac{5}{2\sigma^A_{\min}\rho}\|\hat{g}^s_{t} - \nabla f(x^s_{t})\|^2 - \frac{5}{2\sigma^A_{\min}\rho}\|\nabla f(x^s_{t})\|^2
- \frac{5\sigma^2_{\max}(G)}{2\sigma^A_{\min}\eta^2\rho}\|x^s_{t+1}-x^s_t\|^2 \nonumber \\
& \quad + \frac{\rho}{5}\|Ax^s_{t+1} + \sum_{j=1}^kB_jy_j^{s,t+1} -c\|^2 \nonumber\\
& \geq f(x^s_{t+1}) + \sum_{j=1}^k\psi_j(y_j^{s,t+1}) - \frac{5L^2d}{\sigma^A_{\min}\rho b}\|x^s_t-\tilde{x}^s\|^2_2 - \frac{5L^2 d^2\mu^2}{4\sigma^A_{\min}\rho}
  - \frac{5\delta^2}{2\sigma^A_{\min}\rho} - \frac{5\sigma^2_{\max}(G)}{2\sigma^A_{\min}\eta^2\rho}\|x^s_{t+1}-x^s_t\|^2,
\end{align}
where the first inequality is obtained by applying $ \langle a, b\rangle \leq \frac{1}{2\beta}\|a\|^2 + \frac{\beta}{2}\|b\|^2$ to the terms
$\langle(A^T)^+(\hat{\nabla} f(x_{t}) - \nabla f(x_{t})), Ax_{t+1} + \sum_{j=1}^kB_jy_j^{t+1} -c\rangle$, $\langle(A^T)^+\nabla f(x_{t}), Ax_{t+1} + \sum_{j=1}^kB_jy_j^{t+1} -c\rangle $ and
$\langle(A^T)^+\frac{G}{\eta}(x_{t+1}-x_t), Ax_{t+1} + \sum_{j=1}^kB_jy_j^{t+1} -c\rangle$ with $\beta = \frac{\rho}{5}$, respectively;
the second inequality follows by Lemma 1 of \cite{Huang2019faster} and Assumption 2.
Using the definition of function $R^s_t$ and Assumption 3, we have
\begin{align}
 R^s_{t+1} \geq f^* + \sum_{j=1}^k\psi_j^* - \frac{5L^2d^2\mu^2}{4\sigma^A_{\min}\rho} - \frac{5\delta^2}{2\sigma^A_{\min}\rho}, \ \mbox{for} \ t=0,1,2,\cdots
\end{align}
Thus the function $R^s_t$ is bounded from below. Let $R^*$ denotes a lower bound of $R^s_t$.

Finally, telescoping \eqref{eq:A81} and \eqref{eq:A89} over $t$ from $0$ to $m-1$
and over $s$ from $1$ to $S$, we have
\begin{align}
\frac{1}{T}\sum_{s=1}^S \sum_{t=0}^{m-1} (\sum_{j=1}^k \|y_j^{s,t}-y_j^{s,t+1}\|^2 + \frac{d}{b}\|x^s_t-\tilde{x}^s\|^2_2 + \|x^s_{t+1}-x^s_t\|^2)
\leq \frac{R^1_0 - R^*}{\gamma T}  + \frac{9L^2d^2\mu^2}{\gamma\sigma^A_{\min}\rho} + \frac{L d^2 \mu^2}{4\gamma},
\end{align}
where $\gamma = \min(\sigma_{\min}^H, L, \chi_t)$ and $\chi_t \geq \frac{3\sqrt{71}\kappa_GLd^{l}}{2\alpha} >0 \ (l=0,0.5,1)$.

\end{proof}

Next, based on the above lemmas, we give the convergence analysis of ZO-SVRG-ADMM algorithm. For notational simplicity, let
 \begin{align}
   \nu_1 = k\big(\rho^2\sigma^B_{\max}\sigma^A_{\max} + \rho^2(\sigma^B_{\max})^2 + \sigma^2_{\max}(H)\big), \
   \nu_2 = 6L^2 + \frac{3\sigma^2_{\max}(G)}{\eta^2}, \ \nu_3 = \frac{18L^2 }{\sigma^A_{\min}\rho^2} + \frac{3\sigma^2_{\max}(G)}{\sigma^A_{\min}\eta^2\rho^2}. \nonumber
 \end{align}
\begin{theorem}
Suppose the sequence $\{(x^{s}_t,y_{[k]}^{s,t},\lambda^{s}_t)_{t=1}^m\}_{s=1}^S$ is generated from Algorithm \ref{alg:1}. Let $m=[n^{\frac{1}{3}}]$,
 $b=[d^{1-l} n^{\frac{2}{3}}],\ l \in\{ 0,\frac{1}{2},1\}$, $\eta = \frac{\alpha\sigma_{\min}(G)}{9d^lL} \ (0 < \alpha \leq 1 )$ and
 $\rho = \frac{6\sqrt{71}\kappa_G d^l L}{\sigma^A_{\min}\alpha}$,
then we have
 \begin{align}
\min_{s,t} \mathbb{E}\big[ \mbox{dist}(0,\partial L(x^s_t,y_{[k]}^{s,t},\lambda^s_t))^2\big] \leq O(\frac{d^{2l}}{T}) + O(d^{2+2l}\mu^2),\nonumber
\end{align}
where $\gamma = \min(\sigma_{\min}^H, \chi_{t}, L)$ with $\chi_t \geq \frac{3\sqrt{71}\kappa_Gd^lL}{2\alpha} $, $\nu_{\max}= \max(\nu_2,\nu_3,\nu_4)$ and
 $R^*$ is a lower bound of function $R^s_t$.
It follows that suppose the smoothing parameter $\mu$ and the whole iteration number $T=mS$ satisfy
 \begin{align}
  \mu = O(\frac{\sqrt{\epsilon}}{d^{1+l}}), \quad  T = O(\frac{d^{2l}}{\epsilon}), \nonumber
 \end{align}
 then $(x^{s^*}_{t^*},y_{[k]}^{s^*,t^*},\lambda^{s^*}_{t^*})$
 is an $\epsilon$-approximate solution of \eqref{eq:1}, where $(t^*,s^*) = \mathop{\arg\min}_{t,s}\theta^s_{t}$.
\end{theorem}

\begin{proof}
First, we define a useful variable $\theta^s_{t} = \mathbb{E}\big[\|x^s_{t+1}-x^s_{t}\|^2 + \|x^s_{t}-x^s_{t-1}\|^2 + \frac{d}{b}(\|x^s_{t}-\tilde{x}^s\|^2 + \|x^s_{t-1}-\tilde{x}^s\|^2 )
+ \sum_{j=1}^k \|y_j^{s,t}-y_j^{s,t+1}\|^2 \big]$.
By the step 8 of Algorithm \ref{alg:1}, we have, for all $i\in [k]$
\begin{align} \label{eq:A90}
  \mathbb{E}\big[\mbox{dist}(0,\partial_{y_j} L(x,y_{[k]},\lambda))^2\big]_{s,t+1} & = \mathbb{E}\big[\mbox{dist} (0, \partial \psi_j(y_j^{s,t+1})-B_j^T\lambda^s_{t+1})^2\big] \nonumber \\
 & = \|B_j^T\lambda^s_t -\rho B_j^T(Ax^s_t + \sum_{i=1}^jB_iy_i^{s,t+1} + \sum_{i=j+1}^k B_iy_i^{s,t} -c) - H_j(y_j^{s,t+1}-y_j^{s,t}) -B_j^T\lambda^s_{t+1}\|^2 \nonumber \\
 & = \|\rho B_j^TA(x^s_{t+1}-x^s_{t}) + \rho B_j^T \sum_{i=j+1}^k B_i (y_i^{s,t+1}-y_i^{s,t})- H_j(y_j^{s,t+1}-y_j^{s,t}) \|^2 \nonumber \\
 & \leq k\rho^2\sigma^{B_j}_{\max}\sigma^A_{\max}\|x^s_{t+1}-x^s_t\|^2 + k\rho^2\sigma^{B_j}_{\max}\sum_{i=j+1}^k \sigma^{B_i}_{\max}\|y_i^{s,t+1}-y_i^{s,t}\|^2 \nonumber \\
 & \quad + k\sigma^2_{\max}(H_j)\|y_j^{s,t+1}-y_j^{s,t}\|^2\nonumber \\
 & \leq k\big(\rho^2\sigma^B_{\max}\sigma^A_{\max} + \rho^2(\sigma^B_{\max})^2 + \sigma^2_{\max}(H)\big) \theta^s_{t},
\end{align}
where the first inequality follows by the inequality $\|\frac{1}{n}\sum_{i=1}^n z_i\|^2 \leq \frac{1}{n}\sum_{i=1}^n \|z_i\|^2$.

By the step 9 of Algorithm \ref{alg:1}, we have
\begin{align} \label{eq:A91}
 \mathbb{E}[\mbox{dist}(0,\nabla_x L(x,y_{[k]},\lambda))]_{s,t+1} & = \mathbb{E}\|A^T\lambda^s_{t+1}-\nabla f(x^s_{t+1})\|^2  \nonumber \\
 & = \mathbb{E}\|\hat{g}^s_t - \nabla f(x^s_{t+1}) - \frac{G}{\eta} (x^s_t-x^s_{t+1})\|^2 \nonumber \\
 & = \mathbb{E}\|\hat{g}^s_t - \nabla f(x^s_{t}) +\nabla f(x^s_{t})- \nabla f(x^s_{t+1})
  - \frac{G}{\eta}(x^s_t-x^s_{t+1})\|^2  \nonumber \\
 & \leq  \frac{6 L^2 d}{b}\|x^s_t-\tilde{x}^s\|^2 + 3(L^2+ \frac{\sigma^2_{\max}(G)}{\eta^2})\|x^s_t-x^s_{t+1}\|^2
 + \frac{3L^2d^2\mu^2}{2}  \nonumber \\
 & \leq \big( 6L^2 + \frac{3\sigma^2_{\max}(G)}{\eta^2} \big)\theta^s_{t} + \frac{3L^2d^2\mu^2}{2}.
\end{align}

By the step 10 of Algorithm \ref{alg:1}, we have
\begin{align}\label{eq:A92}
 \mathbb{E}[\mbox{dist}(0,\nabla_{\lambda} L(x,y_{[k]},\lambda))]_{s,t+1} & = \mathbb{E}\|Ax^s_{t+1}+By^s_{t+1}-c\|^2 \nonumber \\
 &= \frac{1}{\rho^2} \mathbb{E} \|\lambda^s_{t+1}-\lambda^s_t\|^2  \nonumber \\
 & \leq \frac{18 L^2d }{\sigma^A_{\min}\rho^2 b} \big( \|x^s_t - \tilde{x}^s\|^2
  + \|x^s_{t-1} - \tilde{x}^s\|^2\big) + \frac{3\sigma^2_{\max}(G)}{\sigma^A_{\min}\eta^2\rho^2}\|x^s_{t+1}-x^s_t\|^2 \nonumber \\
 & \quad + \frac{3(\sigma^2_{\max}(G) + 3L^2\eta^2)}{\sigma^A_{\min}\eta^2\rho^2}\|x^s_{t}-x^s_{t-1}\|^2
  + \frac{9L^2d^2\mu^2}{\sigma^A_{\min}\rho^2}\nonumber \\
 & \leq \big( \frac{18L^2 }{\sigma^A_{\min}\rho^2}
 + \frac{3\sigma^2_{\max}(G)}{\sigma^A_{\min}\eta^2\rho^2} \big)\theta^s_{t} + \frac{9L^2d^2\mu^2}{\sigma^A_{\min}\rho^2}. \nonumber \\
\end{align}

Next, combining the above inequalities \eqref{eq:A90}, \eqref{eq:A91} and \eqref{eq:A92}, we have
\begin{align}
\min_{s,t} \mathbb{E}\big[ \mbox{dist}(0,\partial L(x^s_t,y_{[k]}^{s,t},\lambda^s_t))^2\big]
& \leq \frac{1}{T} \sum_{s=1}^S \sum_{t=1}^m \mathbb{E}\big[ \mbox{dist}(0,\partial L(x^s_t,y_{[k]}^{s,t},\lambda^s_t))^2\big] \nonumber \\
& \leq \frac{\nu_{\max}}{T} \sum_{s=1}^S \sum_{t=0}^{m-1}\theta^s_t
+ \frac{3L^2 d^2 \mu^2}{2} + \frac{9L^2d^2\mu^2}{\sigma^A_{\min}\rho^2} \nonumber \\
& \leq \frac{2\nu_{\max}(R^1_0 - R^*)}{\gamma T}  + \frac{18\nu_{\max}L^2d^2\mu^2}{ \gamma \sigma^A_{\min}\rho} + \frac{\nu_{\max}L d^2 \mu^2}{2\gamma}
+ \frac{3L^2 d^2 \mu^2}{2} + \frac{9L^2d^2\mu^2}{\sigma^A_{\min}\rho^2} \nonumber \\
& = \frac{2\nu_{\max}(R^1_0 - R^*)}{\gamma T}  + \big( \frac{18\nu_{\max}L}{ \gamma \sigma^A_{\min}\rho} + \frac{\nu_{\max}L}{2\gamma}
+ \frac{3L}{2} + \frac{9L}{\sigma^A_{\min}\rho^2} \big)Ld^2\mu^2
\end{align}
where the third inequality holds by Lemma \ref{lem:2}, $\nu_{\max}= \max(\nu_1,\nu_2,\nu_3)$, $\gamma = \min(\sigma_{\min}^H, \chi_{t}, L)$, and $\chi_t \geq \frac{3\sqrt{71}\kappa_GLd^{l}}{2\alpha} >0 \ (l=0,0.5,1)$.

Given $\eta = \frac{\alpha\sigma_{\min}(G)}{9d^lL}\ (0< \alpha \leq 1)$ and $\rho = \frac{6\sqrt{71}\kappa_G Ld^l}{\sigma^A_{\min}\alpha}$,
it is easy verifies that $\gamma = O(1) $ and $\nu_{\max} = O(d^{2l})$, which are independent on $n$ and $d$.
Thus, we obtain
\begin{align}
\min_{s,t} \mathbb{E}\big[ \mbox{dist}(0,\partial L(x^s_t, y_{[k]}^{s,t}, \lambda^s_t))^2\big]  \leq O(\frac{d^{2l}}{T})  + O(d^{2+2l}\mu^2).
\end{align}

\end{proof}

\subsection{ Theoretical Analysis of the ZO-SAGA-ADMM }
In this subsection, we in detail give the convergence analysis of the ZO-SAGA-ADMM algorithm.
We begin with giving some useful lemmas as follows:

\begin{lemma} \label{lem:3}
 Suppose the sequence $\{x_t,y_{[k]}^t,\lambda_t\}_{t=1}^T$ is generated by Algorithm \ref{alg:2}. The following inequality holds
 \begin{align}
 \mathbb{E}\|\lambda_{t+1}-\lambda_{t}\|^2 \leq & \frac{18   L^2d }{\sigma^A_{\min} b} \frac{1}{n}\sum_{i=1}^n \big( \|x_t - z^t_i\|^2
 + \|x_{t-1} - z^{t-1}_i\|^2\big)
  + \frac{3\sigma^2_{\max}(G)}{\sigma^A_{\min}\eta^2}\|x_{t+1}-x_t\|^2 \nonumber \\
  & + \frac{3(\sigma^2_{\max}(G) + 3L^2\eta^2)}{\sigma^A_{\min}\eta^2}\|x_{t}-x_{t-1}\|^2
  + \frac{9L^2d^2\mu^2}{\sigma^A_{\min}}.
 \end{align}
\end{lemma}
\begin{proof}
 By the optimize condition of the the step 7 in Algorithm \ref{alg:2}, we have
 \begin{align}
   \hat{g}_t + \frac{G}{\eta}(x_{t+1}-x_t) - A^T\lambda_t + \rho A^T(Ax_{t+1}+\sum_{j=1}^kB_jy_j^{t+1}-c) = 0.
 \end{align}
 Using the step 8 of Algorithm \ref{alg:2}, then we have
 \begin{align}
  A^T\lambda_{t+1} = \hat{g}_t + \frac{G}{\eta}(x_{t+1}-x_t).
 \end{align}
  It follows that
 \begin{align} \label{eq:A12-1}
  \lambda_{t+1} = (A^T)^+ \big( \hat{g}_t + \frac{G}{\eta}(x_{t+1}-x_t) \big),
 \end{align}
where $(A^T)^+$ is the pseudoinverse of $A^T$. By Assumption 4, \emph{i.e.,} $A$ is a full column matrix, we have $(A^T)^+=A(A^TA)^{-1}$.
Then we have
\begin{align} \label{eq:A12-2}
\mathbb{E} \|\lambda_{t+1}-\lambda_t\|^2 & = \mathbb{E}\|(A^T)^+ \big(  \hat{g}_t - \hat{g}_{t-1} + \frac{G}{\eta}(x_{t+1}-x_t) - \frac{G}{\eta}(x_{t}-x_{t-1})\big)\|^2 \nonumber \\
 & \leq \frac{1}{\sigma^A_{\min}}\big[3 \mathbb{E}\|\hat{g}_t - \hat{g}_{t-1}\|^2 + \frac{3 \mathbb{E}\sigma^2_{\max}(G)}{\eta^2}\|x_{t+1}-x_t\|^2
   + \frac{3\sigma^2_{\max}(G)}{\eta^2} \mathbb{E}\|x_{t}-x_{t-1}\|^2 \big].
\end{align}

Next, considering the upper bound of $\|\hat{g}^s_t - \hat{g}^s_{t-1}\|^2$, we have
\begin{align} \label{eq:A12-3}
  \mathbb{E}\|\hat{g}_t - \hat{g}_{t-1}\|^2 & = \mathbb{E}\|\hat{g}_t - \nabla f(x_t) + \nabla f(x_t) -  \nabla f(x_{t-1}) + \nabla f(x_{t-1}) - \hat{g}_{t-1}\|^2 \nonumber \\
  & \leq 3\mathbb{E}\|\hat{g}_t - \nabla f(x_t)\|^2 + 3\mathbb{E}\|\nabla f(x_t) -  \nabla f(x_{t-1})\|^2 + 3\mathbb{E}\|\nabla f(x_{t-1}) - \hat{g}_{t-1}\|^2 \nonumber \\
  & \leq \frac{6 L^2d}{b} \frac{1}{n} \sum_{i=1}^n \big(\|x_t - z^t_i\|^2 + \|x_{t-1} - x^{t-1}_i\|^2 \big) + 3L^2d\mu^2 +  3\|\nabla f(x_t) -  \nabla f(x_{t-1})\|^2 \nonumber \\
  & \leq \frac{6 L^2d}{b} \frac{1}{n} \sum_{i=1}^n \big(\|x_t - z^t_i\|^2 + \|x_{t-1} - x^{t-1}_i\|^2 \big) + 3L^2\|x_t - x_{t-1}\|^2 + 3L^2d^2\mu^2,
\end{align}
 where the second inequality holds by lemma 3 of \cite{Huang2019faster}, and the third inequality holds by Assumption 1.

Finally, combining the inequalities \eqref{eq:A12-2} and \eqref{eq:A12-3}, we can obtain the above result.
\end{proof}

\begin{lemma} \label{lem:4}
 Suppose the sequence $\{x_t,y_{[k]}^{t},\lambda_t\}_{t=1}^T$ is generated from Algorithm \ref{alg:2},
 and define a \emph{Lyapunov} function
 \begin{align}
 \Omega_t = \mathbb{E}\big[ \mathcal{L}_{\rho} (x_t,y_{[k]}^{t},\lambda_t)+ (\frac{3\sigma^2_{\max}(G)}{\sigma^A_{\min}\rho\eta^2}+\frac{9L^2}{\sigma^A_{\min}\rho})
 \|x_{t}-x_{t-1}\|^2 + \frac{18 L^2d }{\sigma^A_{\min}\rho b}\frac{1}{n} \sum_{i=1}^n\|x_{t-1}-z^{t-1}_i\|^2 + c_t\frac{1}{n} \sum_{i=1}^n\|x_{t}-z^t_i\|^2 \big], \nonumber
\end{align}
 where the positive sequence $\{c_t\}$ satisfies
 \begin{equation*}
  c_t= \left\{
  \begin{aligned}
  & \frac{36L^2d }{\sigma^A_{\min}\rho b} + \frac{2Ld}{b} +(1-p)(1+\beta)c_{t+1}, \ 0 \leq t \leq T-1, \\
  & \\
  & 0, \ t \geq T.
  \end{aligned}
  \right.\end{equation*}
It follows that
\begin{align}
 \frac{1}{T} \sum_{t=1}^T \big( \|x_t-x_{t+1}\|^2 +  \frac{Ld}{b}\frac{1}{n}\sum_{i=1}^n\|x_t-z^t_i\|^2 + \sum_{j=1}^k \|y_j^{t}-y_j^{t+1}\|^2 \big) \leq \frac{\Omega_0 - \Omega^*}{T}
  + \frac{9L^2d^2\mu^2}{\sigma^A_{\min}\rho} + \frac{L d^2 \mu^2}{4},
\end{align}
where $\gamma = \min(\sigma_{\min}^H,L,\chi_t)$ and $\chi_t \geq \frac{3\sqrt{791}\kappa_Gd^l}{2\alpha}\ (l=0,0.5,1)$, and $\Omega^*$ denotes a lower bound of $\Omega_t$.
\end{lemma}

\begin{proof}

By the optimal condition of step 6 in Algorithm \ref{alg:2},
we have, for $j\in [k]$
\begin{align}
0 & =(y_j^{t}-y_j^{t+1})^T\big(\partial \psi_j(y_j^{t+1}) - B_j^T\lambda_t + \rho B_j^T(Ax_t + \sum_{i=1}^jB_iy_i^{t+1} + \sum_{i=j+1}^kB_iy_i^{t}-c) + H_j(y_j^{t+1}-y_j^{t})\big) \nonumber \\
& \leq \psi_j(y_j^{t})- \psi_j(y_j^{t+1}) - (\lambda_t)^T(B_jy_j^{t}-B_jy_j^{t+1}) + \rho(By_j^{t}-By_j^{t+1})^T(Ax_t + \sum_{i=1}^jB_iy_i^{t+1} + \sum_{i=j+1}^kB_iy_i^{t}-c) \nonumber \\
& \quad - \|y_j^{t+1}-y_j^{t}\|^2_{H_j} \nonumber \\
& = \psi_j(y_j^{t})- \psi_j(y_j^{t+1}) - (\lambda_t)^T(Ax_t+\sum_{i=1}^{j-1}B_iy_i^{t+1} + \sum_{i=j}^kB_iy_i^{t}-c) + (\lambda_t)^T(Ax_t+\sum_{i=1}^jB_iy_i^{t+1}+ \sum_{i=j+1}^kB_iy_i^{t}-c) \nonumber \\
& \quad  + \frac{\rho}{2}\|Ax_t +\sum_{i=1}^{j-1}B_iy_i^{t+1} + \sum_{i=j}^kB_iy_i^{t}-c\|^2 - \frac{\rho}{2}\|Ax_t+\sum_{i=1}^jB_iy_i^{t+1}+ \sum_{i=j+1}^kB_iy_i^{t}-c\|^2
  - \|y_j^{t+1}-y_j^{t}\|^2_{H_j} \nonumber \\
& \quad  -\frac{\rho}{2}\|B_jy_j^{t}-B_jy_j^{t+1}\|^2  \nonumber \\
& =\underbrace{ f(x_t) \!+\! \sum_{i=1}^{j}\psi_i(y_i^{t+1}) \!+\! \sum_{i=j+1}^{k}\psi_i(y_l^{t}) \!-\! (\lambda_t)^T(Ax_t+\sum_{i=1}^{j-1}B_iy_i^{t+1} \!+\! \sum_{i=j}^kB_iy_i^{t}-c) \!+\! \frac{\rho}{2}\|Ax_t \!+\!\sum_{i=1}^{j-1}B_iy_i^{t+1} \!+\! \sum_{i=j}^kB_iy_i^{t}-c\|^2}_{\mathcal{L}_{\rho} (x_t,y_{[j-1]}^{t+1},y_{[j:k]}^{t},\lambda_t)} \nonumber \\
& \quad \!-\big(  \underbrace{ f(x_t) \!+\! \sum_{i=1}^{j-1}\psi_i(y_l^{t+1}) \!+\! \sum_{i=j}^{k}\psi_i(y_i^{t}) \!-\! (\lambda_t)^T(Ax_t+\sum_{i=1}^jB_iy_i^{t+1} \!+\! \sum_{i=j+1}^kB_iy_i^{t}-c) \!+\! \frac{\rho}{2}\|Ax_t \!+\! \sum_{i=1}^jB_iy_i^{t+1} \!+\! \sum_{i=j+1}^kB_iy_i^{t}-c\|^2}_{\mathcal{L}_{\rho} (x_t,y_{[j]}^{t+1},y_{[j+1:k]}^{t},\lambda_t)} \big) \nonumber \\
& \quad - \|y_j^{t+1}-y_j^{t}\|^2_{H_j} -\frac{\rho}{2}\|B_jy_j^{t}-B_jy_j^{t+1}\|^2  \nonumber \\
& \leq \mathcal{L}_{\rho} (x_t,y_{[j-1]}^{t+1},y_{[j:k]}^{t},\lambda_t) - \mathcal{L}_{\rho} (x_t,y_{[j]}^{t+1},y_{[j+1:k]}^{t},\lambda_t)
- \sigma_{\min}(H_j)\|y_j^{t}-y_j^{t+1}\|^2,
\end{align}
where the first inequality holds by the convexity of function $\psi_j(y)$,
and the second equality follows by applying the equality
$(a-b)^Tb = \frac{1}{2}(\|a\|^2-\|b\|^2-\|a-b\|^2)$ on the term $(By_j^{t}-By_j^{t+1})^T(Ax_t + \sum_{i=1}^jB_iy_i^{t+1} + \sum_{i=j+1}^kB_iy_i^{t}-c)$.
Thus, we have, for all $j\in[k]$
\begin{align} \label{eq:A13-01}
 \mathcal{L}_{\rho} (x_t,y_{[j-1]}^{t+1},y_{[j:k]}^{t},\lambda_t) \leq \mathcal{L}_{\rho} (x_t,y_{[j]}^{t+1},y_{[j+1:k]}^{t},\lambda_t)
 - \sigma_{\min}(H_j)\|y_j^{t}-y_j^{t+1}\|^2.
\end{align}
Telescoping inequality \eqref{eq:A13-01} over $j$ from $1$ to $k$, we obtain
\begin{align} \label{eq:A13-1}
 \mathcal{L}_{\rho} (x_t,y^{t+1}_{[k]},\lambda_t) \leq \mathcal{L}_{\rho} (x_t,y^{t}_{[k]},\lambda_t)
 - \sigma_{\min}^H\sum_{j=1}^k \|y_j^{t}-y_j^{t+1}\|^2,
\end{align}
where $\sigma_{\min}^H=\min_{j\in[k]}\sigma_{\min}(H_j)$.

By Assumption 1, we have
\begin{align} \label{eq:A13-2}
0 \leq f(x_t) - f(x_{t+1}) + \nabla f(x_t)^T(x_{t+1}-x_t) + \frac{L}{2}\|x_{t+1}-x_t\|^2.
\end{align}
Using the step 7 of Algorithm \ref{alg:2}, we have
\begin{align} \label{eq:A13-3}
 0 = (x_t-x_{t+1})^T \big( \hat{g}_t - A^T\lambda_t + \rho A^T(Ax_{t+1} + \sum_{j=1}^kB_jy_j^{t+1}-c) + \frac{G}{\eta}(x_{t+1}-x_t) \big).
\end{align}
Combining \eqref{eq:A13-2} and \eqref{eq:A13-3}, we have
\begin{align}
 0 & \leq f(x_t) - f(x_{t+1}) + \nabla f(x_t)^T(x_{t+1}-x_t) + \frac{L}{2}\|x_{t+1}-x_t\|^2 \nonumber \\
 & \quad + (x_t-x_{t+1})^T \big( \hat{g}_t - A^T\lambda_t + \rho A^T(Ax_{t+1} + \sum_{j=1}^kB_jy_j^{t+1}-c) + \frac{G}{\eta}(x_{t+1}-x_t) \big)  \nonumber \\
 & = f(x_t) - f(x_{t+1}) + \frac{L}{2}\|x_t-x_{t+1}\|^2 - \frac{1}{\eta}\|x_t - x_{t+1}\|^2_G + (x_t-x_{t+1})^T(\hat{g}_t-\nabla f(x_t)) \nonumber \\
 & \quad -(\lambda_t)^T(Ax_t-Ax_{t+1}) + \rho(Ax_t - Ax_{t+1})^T(Ax_t + \sum_{j=1}^kB_jy_j^{t+1}-c) \nonumber \\
 & \mathop{=}^{(i)} f(x_t) - f(x_{t+1}) + \frac{L}{2}\|x_t-x_{t+1}\|^2 - \frac{1}{\eta}\|x_t - x_{t+1}\|^2_G + (x_t-x_{t+1})^T(\hat{g}_t-\nabla f(x_t)) -(\lambda_t)^T(Ax_t + \sum_{j=1}^kB_jy_j^{t+1}-c)\nonumber \\
 & \quad  + (\lambda_t)^T(Ax_{t+1}+ \sum_{j=1}^kB_jy_j^{t+1}-c) + \frac{\rho}{2}\big(\|Ax_{t} + \sum_{j=1}^kB_jy_j^{t+1}-c\|^2 - \|Ax_{t+1} + \sum_{j=1}^kB_jy_j^{t+1}-c\|^2 - \|Ax_t - Ax_{t+1}\|^2 \big) \nonumber \\
 & = \underbrace{ f(x_t) + \sum_{j=1}^k \psi(y^{t+1}_j) - (\lambda_t)^T(Ax_t + \sum_{j=1}^kB_jy_j^{t+1}-c) +  \frac{\rho}{2} \|Ax_{t} + \sum_{j=1}^kB_jy_j^{t+1}-c\|^2 }_{ \mathcal{L}_{\rho} (x_t,y_{[k]}^{t+1},\lambda_t)} \nonumber \\
 & \quad - \big( \underbrace{f(x_{t+1}) + \sum_{j=1}^k \psi(y^{t+1}_j) -(\lambda_t)^T(Ax_{t+1} + \sum_{j=1}^kB_jy_j^{t+1}-c) +  \frac{\rho}{2} \|Ax_{t+1} + \sum_{j=1}^kB_jy_j^{t+1}-c\|^2}_{ \mathcal{L}_{\rho} (x_{t+1},y_{[k]}^{t+1},\lambda_t)} \big) \nonumber \\
 & \quad  + \frac{L}{2}\|x_t-x_{t+1}\|^2 + (x_t-x_{t+1})^T(\hat{g}_t-\nabla f(x_t)) - \frac{1}{\eta}\|x_t - x_{t+1}\|^2_G - \frac{\rho}{2}\|Ax_t - Ax_{t+1}\|^2 \nonumber \\
 & \leq \mathcal{L}_{\rho} (x_t,y_{[k]}^{t+1},\lambda_t) -  \mathcal{L}_{\rho} (x_{t+1},y_{[k]}^{t+1},\lambda_t)
 - (\frac{\sigma_{\min}(G)}{\eta} + \frac{\rho \sigma^A_{\min}}{2} - \frac{L}{2}) \|x_t - x_{t+1}\|^2 +(x_t-x_{t+1})^T(\hat{g}_t-\nabla f(x_t)) \nonumber \\
 & \mathop{\leq}^{(ii)}  \mathcal{L}_{\rho} (x_t,y_{[k]}^{t+1},\lambda_t) -  \mathcal{L}_{\rho} (x_{t+1},y_{[k]}^{t+1},\lambda_t)
 - (\frac{\sigma_{\min}(G)}{\eta} + \frac{\rho \sigma^A_{\min}}{2} - L) \|x_t - x_{t+1}\|^2 + \frac{1}{2L}\|\hat{g}_t-\nabla f(x_t)\|^2 \nonumber \\
 & \mathop{\leq}^{(iii)} \mathcal{L}_{\rho} (x_t,y_{[k]}^{t+1},\lambda_t) -  \mathcal{L}_{\rho} (x_{t+1},y_{[k]}^{t+1},\lambda_t)
 - (\frac{\sigma_{\min}(G)}{\eta} + \frac{\rho \sigma^A_{\min}}{2} - L) \|x_t - x_{t+1}\|^2 +\frac{Ld}{b}\frac{1}{n} \sum_{i=1}^n \|x_t-z^t_i\|^2 + \frac{L d^2 \mu^2}{4},
\end{align}
where the equality $(i)$ holds by applying the equality
$(a-b)^Tb = \frac{1}{2}(\|a\|^2-\|b\|^2-\|a-b\|^2)$ on the
term $(Ax_t - Ax_{t+1})^T(Ax_{t+1}+\sum_{j=1}^kB_jy_j^{t+1}-c)$; the inequality
$(ii)$ follows by the inequality $a^Tb \leq \frac{L}{2}\|a\|^2 + \frac{1}{2L}\|a\|^2$,
and the inequality $(iii)$ holds by lemma 3 of \cite{Huang2019faster}.
Thus, we obtain
\begin{align} \label{eq:A13-4}
\mathcal{L}_{\rho} (x_{t+1},y_{[k]}^{t+1},\lambda_t) \leq & \mathcal{L}_{\rho} (x_t,y_{[k]}^{t+1},\lambda_t) -
(\frac{\sigma_{\min}(G)}{\eta} + \frac{\rho \sigma^A_{\min}}{2} - L)\|x_t - x_{t+1}\|^2 \nonumber \\
& +\frac{ L d}{b}\frac{1}{n} \sum_{i=1}^n\|x_t-z^t_i\|^2 + \frac{L d^2 \mu^2}{4}.
\end{align}

By the step 8 in Algorithm \ref{alg:2}, we have
\begin{align} \label{eq:A13-5}
\mathcal{L}_{\rho} (x_{t+1},y_{[k]}^{t+1},\lambda_{t+1}) - \mathcal{L}_{\rho} (x_{t+1},y_{[k]}^{t+1},\lambda_t)
& = \frac{1}{\rho}\|\lambda_{t+1}-\lambda_t\|^2 \nonumber \\
& \leq  \frac{18   L^2d  }{\sigma^A_{\min}\rho b} \frac{1}{n} \sum_{i=1}^n\big( \|x_t - z^t_i\|^2 + \|x_{t-1} - z^{t-1}_i\|^2\big)
 + \frac{3\sigma^2_{\max}(G)}{\sigma^A_{\min}\eta^2\rho}\|x_{t+1}-x_t\|^2 \nonumber \\
& \quad + \frac{3(\sigma^2_{\max}(G)+3L^2\eta^2)}{\sigma^A_{\min}\eta^2\rho}\|x_{t}-x_{t-1}\|^2 + \frac{9L^2d^2\mu^2}{\sigma^A_{\min}\rho}.
\end{align}
Combining \eqref{eq:A13-1}, \eqref{eq:A13-4} and \eqref{eq:A13-5}, we have
\begin{align} \label{eq:A13-6}
\mathcal{L}_{\rho} (x_{t+1},y_{[k]}^{t+1},\lambda_{t+1}) & \leq \mathcal{L}_{\rho} (x_t,y_{[k]}^{t},\lambda_t)
 - (\frac{\sigma_{\min}(G)}{\eta} + \frac{\rho \sigma^A_{\min}}{2} - L)\|x_t - x_{t+1}\|^2 - \sigma_{\min}^H\sum_{j=1}^k \|y_j^{t}-y_j^{t+1}\|^2\nonumber \\
& \quad +\frac{ L d}{b}\frac{1}{n} \sum_{i=1}^n\|x_t-z^t_i\|^2 + \frac{18L^2d }{\sigma^A_{\min}\rho b}
\frac{1}{n} \sum_{i=1}^n\big( \|x_t - z^t_i\|^2 + \|x_{t-1} - z^{t-1}_i\|^2\big)
+ \frac{3\sigma^2_{\max}(G)}{\sigma^A_{\min}\eta^2\rho}\|x_{t+1}-x_t\|^2 \nonumber \\
& \quad + \frac{3(\sigma^2_{\max}(G)+3L^2\eta^2)}{\sigma^A_{\min}\eta^2\rho}\|x_{t}-x_{t-1}\|^2
   + \frac{9L^2d^2\mu^2}{\sigma^A_{\min}\rho} + \frac{L d^2 \mu^2}{4}.
\end{align}

Next, we define a \emph{Lyapunov} function as follows:
\begin{align}
 \Omega_t = \mathbb{E}\big[ \mathcal{L}_{\rho} (x_t,y_{[k]}^{t},\lambda_t)+ (\frac{3\sigma^2_{\max}(G)}{\sigma^A_{\min}\rho\eta^2}+\frac{9L^2}{\sigma^A_{\min}\rho})
 \|x_{t}-x_{t-1}\|^2 + \frac{18 L^2d }{\sigma^A_{\min}\rho b}\frac{1}{n} \sum_{i=1}^n\|x_{t-1}-z^{t-1}_i\|^2 + c_t\frac{1}{n} \sum_{i=1}^n\|x_{t}-z^t_i\|^2 \big]. \nonumber
\end{align}

By the step 9 of Algorithm \ref{alg:2}, we have
 \begin{align} \label{eq:A13-7}
  \frac{1}{n}\sum_{i=1}^n \|x_{t+1}-z^{t+1}_i\|^2 &= \frac{1}{n}\sum_{i=1}^n \big( p\|x_{t+1}-x_{t}\|^2 + (1-p)\|x_{t+1}-z^{t}_i\|^2 \big) \nonumber \\
   & = \frac{p}{n} \sum_{i=1}^n \|x_{t+1}-x_{t}\|^2 + \frac{1-p}{n}\sum_{i=1}^n\|x_{t+1}-z^{t}_i\|^2  \nonumber \\
   & = p \|x_{t+1}-x_{t}\|^2 + \frac{1-p}{n} \sum_{i=1}^n \|x_{t+1}-z^{t}_i\|^2,
 \end{align}
where $p$ denotes probability of an index $i$ being in $\mathcal{I}_t$. Here, we have
 \begin{align}
  p = 1-(1-\frac{1}{n})^b \geq 1- \frac{1}{1+b/n} = \frac{b/n}{1+b/n} \geq \frac{b}{2n},
 \end{align}
where the first inequality follows from $(1-a)^b\leq \frac{1}{1+ab}$, and the second inequality holds by $b\leq n$.
Considering the upper bound of $\|x_{t+1}-z^{t}_i\|^2$, we have
\begin{align} \label{eq:A13-8}
  \|x_{t+1}-z^{t}_i\|^2 & = \|x_{t+1}-x_t+x_t-z^{t}_i\|^2 \nonumber \\
  & =  \|x_{t+1}-x_t\|^2 + 2(x_{t+1}-x_t)^T(x_t-z^{t}_i)+ \|x_t-z^{t}_i\|^2 \nonumber \\
  & \leq \|x_{t+1}-x_t\|^2 + 2\big( \frac{1}{2\beta}\|x_{t+1}-x_t\|^2 + \frac{\beta}{2}\|x_t-z^{t}_i\|^2\big)+ \|x_t-z^{t}_i\|^2 \nonumber \\
  & = (1+\frac{1}{\beta})\|x_{t+1}-x_t\|^2 + (1+\beta)\|x_t-z^{t}_i\|^2,
\end{align}
where $\beta>0$.
Combining \eqref{eq:A13-7} with \eqref{eq:A13-8}, we have
 \begin{align}
 \frac{1}{n}\sum_{i=1}^n \|x_{t+1}-z^{t+1}_i\|^2 \leq (1+\frac{1-p}{\beta})\|x_{t+1}-x_t\|^2 +\frac{(1-p)(1+\beta)}{n}\sum_{i=1}^n\|x_t-z^{t}_i\|^2.
 \end{align}
It follows that
\begin{align} \label{eq:A13-9}
\Omega_{t+1} & = \mathbb{E}\big[\mathcal{L}_{\rho}(x_{t+1},y_{[k]}^{t+1},\lambda_{t+1}) + (\frac{3\sigma^2_{\max}(G)}{\sigma^A_{\min}\rho\eta^2}+\frac{9L^2}{\sigma^A_{\min}\rho})\|x_{t+1}-x_{t}\|^2
+ \frac{18L^2d }{\sigma^A_{\min} b\rho}\frac{1}{n} \sum_{i=1}^n\|x_{t}-z^{t}_i\|^2 + c_{t+1}\frac{1}{n}\sum_{i=1}^n \|x_{t+1}-z^{t+1}_i\|^2\big] \nonumber \\
& \leq \mathcal{L}_{\rho} (x_t,y_{[k]}^{t},\lambda_t) + (\frac{3\sigma^2_{\max}(G)}{\sigma^A_{\min}\rho\eta^2}+\frac{9L^2}{\sigma^A_{\min}\rho})\|x_{t}-x_{t-1}\|^2
 + \frac{18L^2d}{\sigma^A_{\min}\rho b} \frac{1}{n}\sum_{i=1}^n\|x_{t-1}-z^{t-1}_i\|^2   \nonumber \\
& \quad + \big(\frac{36L^2d }{\sigma^A_{\min}\rho b} + \frac{2L d}{b}+(1-p)(1+\beta)c_{t+1}\big)\frac{1}{n}\sum_{i=1}^n\|x_t-z^t_i\|^2
+ \frac{9L^2d^2\mu^2}{\sigma^A_{\min}\rho} + \frac{L d^2 \mu^2}{4} - \frac{L d}{b}\frac{1}{n}\sum_{i=1}^n\|x_t-z^t_i\|^2 \nonumber\\
& \quad - \big( \underbrace{ \frac{\sigma_{\min}(G)}{\eta}+\frac{\rho\sigma^A_{\min}}{2} - L
-\frac{6\sigma^2_{\max}(G)}{\sigma^A_{\min}\eta^2\rho} - \frac{9L^2}{\sigma^A_{\min}\rho}-(1+\frac{1-p}{\beta})c_{t+1} }_{\chi_t}\big)\|x_t - x_{t+1}\|^2
 - \sigma_{\min}^H\sum_{j=1}^k \|y_j^{t}-y_j^{t+1}\|^2 \nonumber \\
& = \Omega_t - \chi_t \|x_t - x_{t+1}\|^2 - \frac{L d}{b}\frac{1}{n}\sum_{i=1}^n\|x_t-z^t_i\|^2 - \sigma_{\min}^H\sum_{j=1}^k \|y_j^{t}-y_j^{t+1}\|^2 + \frac{9L^2d^2\mu^2}{\sigma^A_{\min}\rho} + \frac{L d^2 \mu^2}{4},
\end{align}
where $c_t = \frac{36L^2d }{\sigma^A_{\min}\rho b} + \frac{2Ld}{b} +(1-p)(1+\beta)c_{t+1}$.

Let $c_T = 0$ and $\beta=\frac{b}{4n}$. Since $(1-p)(1+\beta)=1+\beta-p-p\beta\leq 1+\beta-p$ and $p\geq \frac{b}{2n}$,
it follows that
\begin{align}
 c_t \leq c_{t+1}(1-\theta) + \frac{36L^2d }{\sigma^A_{\min} b\rho} + \frac{2L d}{b},
\end{align}
where $\theta = p-\beta\geq \frac{b}{4n}$.
Then recursing on $t$, for $0\leq t \leq T-1$, we have
\begin{align}
 c_t \leq \frac{2d}{b}(\frac{18L^2}{\sigma^A_{\min}\rho} + L) \frac{1-\theta^{T-t}}{\theta} \leq \frac{2d}{b\theta}(\frac{18L^2}{\sigma^A_{\min}\rho} + L)
 \leq \frac{8nd}{b^2}(\frac{18L^2}{\sigma^A_{\min}\rho} + L).
\end{align}
It follows that
\begin{align}
\chi_t & = \frac{\sigma_{\min}(G)}{\eta}+\frac{\rho\sigma^A_{\min}}{2} - L
-\frac{6\sigma^2_{\max}(G)}{\sigma^A_{\min}\eta^2\rho} - \frac{9L^2}{\sigma^A_{\min}\rho}-(1+\frac{1-p}{\beta})c_{t+1} \nonumber \\
& \geq \frac{\sigma_{\min}(G)}{\eta}+\frac{\rho\sigma^A_{\min}}{2} - L
-\frac{6\sigma^2_{\max}(G)}{\sigma^A_{\min}\eta^2\rho} - \frac{9L^2}{\sigma^A_{\min}\rho}-( 1 + \frac{4n-2b}{b})\frac{8nd}{b^2}(\frac{18L^2}{\sigma^A_{\min}\rho} + L) \nonumber \\
& = \frac{\sigma_{\min}(G)}{\eta}+\frac{\rho\sigma^A_{\min}}{2} - L
-\frac{6\sigma^2_{\max}(G)}{\sigma^A_{\min}\eta^2\rho} - \frac{9L^2}{\sigma^A_{\min}\rho} - (\frac{4n}{b}-1)\frac{8nd}{b^2}(\frac{18L^2}{\sigma^A_{\min}\rho} + L) \nonumber \\
& \geq \frac{\sigma_{\min}(G)}{\eta}+\frac{\rho\sigma^A_{\min}}{2} - L
-\frac{6\sigma^2_{\max}(G)}{\sigma^A_{\min}\eta^2\rho} - \frac{9L^2}{\sigma^A_{\min}\rho} - \frac{32n^2d}{b^3}(\frac{18L^2}{\sigma^A_{\min}\rho} + L) \nonumber \\
& = \underbrace{\frac{\sigma_{\min}(G)}{\eta}- L - \frac{32n^2dL}{b^3}}_{T_1} + \underbrace{\frac{\rho\sigma^A_{\min}}{2} -\frac{6\sigma^2_{\max}(G)}{\sigma^A_{\min}\eta^2\rho} - \frac{9L^2}{\sigma^A_{\min}\rho} - \frac{576n^2dL^2}{\sigma^A_{\min}\rho b^3}}_{T_2}
\end{align}

When $1\leq d < n $, let $ b=d^{\frac{1}{3}}n^{\frac{2}{3}}$ (i.e., $ b=d^{\frac{1-l}{3}}n^{\frac{2}{3}},\ l=0$) and $0< \eta \leq \frac{\sigma_{\min}(G)}{33L}$, we have $T_1 \geq 0$.
Further, let $\eta = \frac{\alpha\sigma_{\min}(G)}{33L} \ (0 < \alpha \leq 1)$ and $\rho = \frac{6\sqrt{791}\kappa_G L}{\sigma^A_{\min}\alpha}$,
we have
\begin{align}
 T_2 & = \frac{\rho\sigma^A_{\min}}{2} -\frac{6\sigma^2_{\max}(G)}{\sigma^A_{\min}\eta^2\rho} - \frac{9L^2}{\sigma^A_{\min}\rho} - \frac{576n^2dL^2}{\sigma^A_{\min}\rho b^3} \nonumber \\
 & = \frac{\rho\sigma^A_{\min}}{2} -\frac{6534\kappa_G^2L^2}{\sigma^A_{\min}\rho\alpha^2} - \frac{9L^2}{\sigma^A_{\min}\rho}-\frac{576L^2}{\sigma^A_{\min}\rho} \nonumber \\
 & \geq \frac{\rho\sigma^A_{\min}}{2} -\frac{7119\kappa_G^2L^2}{\sigma^A_{\min}\rho\alpha^2} \nonumber \\
 & = \frac{\rho\sigma^A_{\min}}{4} + \underbrace{\frac{\rho\sigma^A_{\min}}{4} -\frac{7119\kappa_G^2L^2}{\sigma^A_{\min}\rho\alpha^2}}_{\geq 0} \nonumber \\
 & \geq \frac{3\sqrt{791}\kappa_G L}{2\alpha}.
\end{align}
Thus, we have $\chi_t \geq \frac{3\sqrt{791}\kappa_G L}{2\alpha}$.

When $n \leq d < n^2$, let $b= d^{\frac{1}{6}}n^{\frac{2}{3}}$ (i.e., $ b=d^{\frac{1-l}{3}}n^{\frac{2}{3}},\ l=0.5$) and $0< \eta \leq \frac{\sigma_{\min}(G)}{33\sqrt{d}L}$, we have $T_1 \geq 0$.
Further, let $\eta = \frac{\alpha\sigma_{\min}(G)}{33\sqrt{d}L} \ (0 < \alpha \leq 1)$ and $\rho = \frac{6\sqrt{791d}\kappa_GL}{\sigma^A_{\min}\alpha}$,
we have
\begin{align}
 T_2 & = \frac{\rho\sigma^A_{\min}}{2} -\frac{6\sigma^2_{\max}(G)}{\sigma^A_{\min}\eta^2\rho} - \frac{9L^2}{\sigma^A_{\min}\rho} - \frac{576n^2dL^2}{\sigma^A_{\min}\rho b^3} \nonumber \\
 & = \frac{\rho\sigma^A_{\min}}{2} -\frac{6534\kappa_G^2L^2d}{\sigma^A_{\min}\rho\alpha^2} - \frac{9L^2}{\sigma^A_{\min}\rho}-\frac{576L^2\sqrt{d}}{\sigma^A_{\min}\rho} \nonumber \\
 & \geq \frac{\rho\sigma^A_{\min}}{2} -\frac{7119\kappa_G^2L^2d}{\sigma^A_{\min}\rho\alpha^2} \nonumber \\
 & = \frac{\rho\sigma^A_{\min}}{4} + \underbrace{\frac{\rho\sigma^A_{\min}}{4} -\frac{7119\kappa_G^2L^2d}{\sigma^A_{\min}\rho\alpha^2}}_{\geq 0} \nonumber \\
 & \geq \frac{3\sqrt{791d}\kappa_G}{2\alpha}.
\end{align}
Thus, we have $\chi_t \geq \frac{3\sqrt{791d}\kappa_GL}{2\alpha}$.

When $n^2 \leq d $, let $b=n^{\frac{2}{3}}$ (i.e.,$ b=d^{\frac{1-l}{3}}n^{\frac{2}{3}},\ l=1$) and $0< \eta \leq \frac{\sigma_{\min}(G)}{33dL}$, we have $T_1 \geq 0$.
Further, let $\eta = \frac{\alpha\sigma_{\min}(G)}{33dL} \ (0 < \alpha \leq 1)$ and $\rho = \frac{6\sqrt{791}\kappa_Gd}{\sigma^A_{\min}\alpha}$,
we have
\begin{align}
 T_2 & = \frac{\rho\sigma^A_{\min}}{2} -\frac{6\sigma^2_{\max}(G)}{\sigma^A_{\min}\eta^2\rho} - \frac{9L^2}{\sigma^A_{\min}\rho} - \frac{576n^2dL^2}{\sigma^A_{\min}\rho b^3} \nonumber \\
 & = \frac{\rho\sigma^A_{\min}}{2} -\frac{6534\kappa_G^2L^2d^2}{\sigma^A_{\min}\rho\alpha^2} - \frac{9L^2}{\sigma^A_{\min}\rho}-\frac{576L^2d}{\sigma^A_{\min}\rho} \nonumber \\
 & \geq \frac{\rho\sigma^A_{\min}}{2} -\frac{7119\kappa_G^2L^2d^2}{\sigma^A_{\min}\rho\alpha^2} \nonumber \\
 & = \frac{\rho\sigma^A_{\min}}{4} + \underbrace{\frac{\rho\sigma^A_{\min}}{4} -\frac{7119\kappa_G^2L^2d^2}{\sigma^A_{\min}\rho\alpha^2}}_{\geq 0} \nonumber \\
 & \geq \frac{3\sqrt{791}\kappa_Gd}{2\alpha}.
\end{align}
Thus, we have $\chi_t \geq \frac{3\sqrt{791}\kappa_Gd}{2\alpha}$.

By Assumption 4, i.e., $A$ is a full column rank matrix,
we have $(A^T)^+ = A(A^T A)^{-1}$.
It follows that $\sigma_{\max}((A^T)^+)^T(A^T)^+) = \sigma_{\max}((A^TA)^{-1}) = \frac{1}{\sigma_{\min}^A}$.
Since $\lambda_{t+1} = (A^T)^+ \big( \hat{g}_t + \frac{G}{\eta}(x_{t+1}-x_t) \big)$,
we have
\begin{align}
& \mathcal{L}_{\rho} (x_{t+1},y_{[k]}^{t+1},\lambda_{t+1})
= f(x_{t+1}) + \sum_{j=1}^k\psi_j(y_j^{t+1}) - \lambda_{t+1}^T(Ax_{t+1} + \sum_{j=1}^kB_jy_j^{t+1} - c) + \frac{\rho}{2}\|Ax_{t+1} + \sum_{j=1}^kB_jy_j^{t+1} -c\|^2 \nonumber \\
& = f(x_{t+1}) + \sum_{j=1}^k\psi_j(y_j^{t+1}) -  \langle(A^T)^+(\hat{g}_{t} + \frac{G}{\eta}(x_{t+1}-x_t)), Ax_{t+1} + \sum_{j=1}^kB_jy_j^{t+1} -c\rangle  + \frac{\rho}{2}\|Ax_{t+1}
+ \sum_{j=1}^kB_jy_j^{t+1}-c\|^2 \nonumber \\
& = f(x_{t+1}) + \sum_{j=1}^k\psi_j(y_j^{t+1}) - \langle(A^T)^+(\hat{g}_{t} - \nabla f(x_{t}) + \nabla f(x_{t})+ \frac{G}{\eta}(x_{t+1}-x_t)), Ax_{t+1} + \sum_{j=1}^kB_jy_j^{t+1} -c\rangle  \nonumber \\
& \quad +  \frac{\rho}{2}\|Ax_{t+1} + \sum_{j=1}^kB_jy_j^{t+1} -c\|^2 \nonumber \\
& \geq f(x_{t+1}) + \sum_{j=1}^k\psi_j(y_j^{t+1}) - \frac{5}{2\sigma^A_{\min}\rho}\|\hat{g}_{t} - \nabla f(x_{t})\|^2 - \frac{5}{2\sigma^A_{\min}\rho}\|\nabla f(x_{t})\|^2
- \frac{5\sigma^2_{\max}(G)}{2\sigma^A_{\min}\eta^2\rho}\|x_{t+1}-x_t\|^2  \nonumber \\
& \quad + \frac{\rho}{5}\|Ax_{t+1} + \sum_{j=1}^kB_jy_j^{t+1} -c\|^2 \nonumber\\
& \geq f(x_{t+1}) + \sum_{j=1}^k\psi_j(y_j^{t+1}) - \frac{5L^2d}{\sigma^A_{\min}\rho b}\frac{1}{n} \sum_{i=1}^n \|x_t-z^t_i\|^2_2 - \frac{5L^2 d^2\mu^2}{4\sigma^A_{\min}\rho}
 - \frac{5\delta^2}{2\sigma^A_{\min}\rho} - \frac{5\sigma^2_{\max}(G)}{2\sigma^A_{\min}\eta^2\rho}\|x_{t+1}-x_t\|^2
\end{align}
where the first inequality is obtained by applying $ \langle a, b\rangle \leq \frac{1}{2\beta}\|a\|^2 + \frac{\beta}{2}\|b\|^2$ to the terms
$\langle(A^T)^+(\hat{\nabla} f(x_{t}) - \nabla f(x_{t})), Ax_{t+1} + \sum_{j=1}^kB_jy_j^{t+1} -c\rangle$, $\langle(A^T)^+\nabla f(x_{t}), Ax_{t+1} + \sum_{j=1}^kB_jy_j^{t+1} -c\rangle $ and
$\langle(A^T)^+\frac{G}{\eta}(x_{t+1}-x_t), Ax_{t+1} + \sum_{j=1}^kB_jy_j^{t+1} -c\rangle$ with $\beta = \frac{\rho}{5}$, respectively; the second inequality follows by Lemma 3 of \cite{Huang2019faster}
and Assumption 2.
By definition of the function $\Omega_t$ and Assumption 3, we have
\begin{align}
 \Omega_{t+1} \geq f^* + \sum_{j=1}^k\psi_j^* - \frac{5L^2d^2\mu^2}{4\sigma^A_{\min}\rho} - \frac{5\delta^2}{2\sigma^A_{\min}\rho}, \ \mbox{for} \ t=0,1,2,\cdots
\end{align}
Thus, the function $\Omega_t$ is bounded from below. Let $\Omega^*$ denotes a lower bound of $\Omega_t$.

Finally, telescoping inequality \eqref{eq:A13-9} over $t$ from $0$ to $T$,
we have
\begin{align}
 \frac{1}{T} \sum_{t=1}^T \big( \|x_t-x_{t+1}\|^2 +  \frac{Ld}{b}\frac{1}{n}\sum_{i=1}^n\|x_t-z^t_i\|^2 + \sum_{j=1}^k \|y_j^{t}-y_j^{t+1}\|^2 \big) \leq \frac{\Omega_0 - \Omega^*}{T}
  + \frac{9L^2d^2\mu^2}{\sigma^A_{\min}\rho} + \frac{L d^2 \mu^2}{4},
\end{align}
where $\gamma = \min(\sigma_{\min}^H,L,\chi_t)$ and $\chi_t \geq \frac{3\sqrt{791}\kappa_Gd^l}{2\alpha}\ (l=0,0.5,1)$.

\end{proof}

Next, based on the above lemmas, we give the convergence properties of the ZO-SAGA-ADMM algorithm.
For notational simplicity, let
 \begin{align}
 \nu_1 = k\big(\rho^2\sigma^B_{\max}\sigma^A_{\max} + \rho^2(\sigma^B_{\max})^2 + \sigma^2_{\max}(H)\big), \
  \nu_2 = 6L^2 + \frac{3\sigma^2_{\max}(G)}{\eta^2}, \ \nu_3= \frac{18L^2 }{\sigma^A_{\min}\rho^2} + \frac{3\sigma^2_{\max}(G)}{\sigma^A_{\min}\eta^2\rho^2}. \nonumber
 \end{align}
\begin{theorem}
Suppose the sequence $\{x_t,y_{[k]}^{t},\lambda_t\}_{t=1}^T$ is generated from Algorithm \ref{alg:2}. Let $b = n^{\frac{2}{3}}d^{\frac{1-l}{3}},\ l \in\{ 0,\frac{1}{2},1\}$,
$\eta = \frac{\alpha\sigma_{\min}(G)}{33d^lL} \ (0 < \alpha \leq 1)$ and $\rho = \frac{6\sqrt{791}\kappa_G d^l L}{\sigma^A_{\min}\alpha}$
then we have
\begin{align}
\min_{1\leq t \leq T} \mathbb{E}\big[ \mbox{dist}(0,\partial L(x_t,y_{[k]}^{t},\lambda_t))^2\big] \leq O(\frac{d^{2l}}{T}) + O(d^{2+2l}\mu^2),\nonumber
\end{align}
where $\gamma = \min(\sigma_{\min}^H, \chi_{t}, L)$ with $\chi_t \geq \frac{3\sqrt{791}\kappa_G d^{l} L}{2\alpha}$, $\nu_{\max}= \max(\nu_2,\nu_3,\nu_4)$ and
$\Omega^*$ is a lower bound of function $\Omega_t$.
It follows that suppose the parameters $\mu$ and $T$ satisfy
\begin{align}
 \mu = O(\frac{\sqrt{\epsilon}}{d^{1+l}}), \quad  T = O(\frac{d^{2l}}{\epsilon}), \nonumber
\end{align}
then $(x_{t^*},y_{[k]}^{t^*},\lambda_{t^*})$ is an $\epsilon$-approximate solution of \eqref{eq:1},
where $t^* = \mathop{\arg\min}_{ 1\leq t\leq T}\theta_{t}$.
\end{theorem}

\begin{proof}
We begin with defining an useful variable $\theta_t = \mathbb{E}\big[\|x_{t+1}-x_t\|^2 + \|x_t-x_{t-1}\|^2 + \frac{d}{b n}\sum^n_{i=1} (\|x_t-z^t_i\|^2 + \|x_{t-1}-z^{t-1}_i\|^2)
 + \sum_{j=1}^k \|y_j^{t}-y_j^{t+1}\|^2 \big]$.
By the optimal condition of the step 6 in Algorithm \ref{alg:2}, we have, for all $i\in [k]$
\begin{align} \label{eq:A79}
  \mathbb{E}\big[\mbox{dist}(0,\partial_{y_j} L(x,y_{[k]},\lambda))^2\big]_{t+1} & = \mathbb{E}\big[\mbox{dist} (0, \partial \psi_j(y_j^{t+1})-B_j^T\lambda_{t+1})^2\big] \nonumber \\
 & = \|B_j^T\lambda_t -\rho B_j^T(Ax_t + \sum_{i=1}^jB_iy_i^{t+1} + \sum_{i=j+1}^k B_iy_i^{t} -c) - H_j(y_j^{t+1}-y_j^{t}) -B_j^T\lambda_{t+1}\|^2 \nonumber \\
 & = \|\rho B_j^TA(x_{t+1}-x_{t}) + \rho B_j^T \sum_{i=j+1}^k B_i (y_i^{t+1}-y_i^{t})- H_j(y_j^{t+1}-y_j^{t}) \|^2 \nonumber \\
 & \leq k\rho^2\sigma^{B_j}_{\max}\sigma^A_{\max}\|x_{t+1}-x_t\|^2 + k\rho^2\sigma^{B_j}_{\max}\sum_{i=j+1}^k \sigma^{B_i}_{\max}\|y_i^{t+1}-y_i^{t}\|^2 \nonumber \\
 & \quad + k\sigma^2_{\max}(H_j)\|y_j^{t+1}-y_j^{t}\|^2\nonumber \\
 & \leq k\big(\rho^2\sigma^B_{\max}\sigma^A_{\max} + \rho^2(\sigma^B_{\max})^2 + \sigma^2_{\max}(H)\big) \theta_{t},
\end{align}
where the first inequality follows by the inequality $\|\sum_{i=1}^r \alpha_i\|^2 \leq r\sum_{i=1}^r \|\alpha_i\|^2$.

By the step 7 in Algorithm \ref{alg:2}, we have
\begin{align} \label{eq:A80}
 \mathbb{E}[\mbox{dist}(0,\nabla_x L(x,y_{[k]},\lambda))]_{t+1} & = \mathbb{E}\|A^T\lambda_{t+1}-\nabla f(x_{t+1})\|^2  \nonumber \\
 & = \mathbb{E}\|\hat{g}_t - \nabla f(x_{t+1}) - \frac{G}{\eta} (x_t-x_{t+1})\|^2 \nonumber \\
 & = \mathbb{E}\|\hat{g}_t - \nabla f(x_{t}) +\nabla f(x_{t})- \nabla f(x_{t+1})
  - \frac{G}{\eta}(x_t-x_{t+1})\|^2  \nonumber \\
 & \leq  \frac{6  L^2 d}{b n}\sum_{i=1}^n\|x_t-z^t_i\|^2 + 3(L^2+ \frac{\sigma^2_{\max}(G)}{\eta^2})\|x_t-x_{t+1}\|^2
 + \frac{3L^2d^2\mu^2}{2}  \nonumber \\
 & \leq \big( 6L^2+ \frac{3\sigma^2_{\max}(G)}{\eta^2} \big)\theta_{t} + \frac{3L^2d^2\mu^2}{2}.
\end{align}

By the step 8 of Algorithm \ref{alg:2}, we have
\begin{align} \label{eq:A81}
 \mathbb{E}[\mbox{dist}(0,\nabla_{\lambda} L(x,y_{[k]},\lambda))]_{t+1} & = \mathbb{E}\|Ax_{t+1}+\sum_{j=1}^kB_jy_j^{t+1}-c\|^2 \nonumber \\
 &= \frac{1}{\rho^2} \mathbb{E} \|\lambda_{t+1}-\lambda_t\|^2  \nonumber \\
 & \leq \frac{18L^2d }{\sigma^A_{\min}\rho^2 b n} \sum_{i=1}^n\big( \|x_t - z^t_i\|^2
  + \|x_{t-1} - z^{t-1}_i\|^2\big) + \frac{3\sigma^2_{\max}(G)}{\sigma^A_{\min}\eta^2\rho^2}\|x_{t+1}-x_t\|^2 \nonumber \\
 & \quad +(\frac{3\sigma^2_{\max}(G)}{\sigma^A_{\min}\eta^2\rho^2} + \frac{9L^2)}{\sigma^A_{\min}\rho^2})\|x_{t}-x_{t-1}\|^2
  + \frac{9L^2d\mu^2}{\sigma^A_{\min}\rho^2}\nonumber \\
 & \leq \big( \frac{18L^2 }{\sigma^A_{\min}\rho^2}
 + \frac{3\sigma^2_{\max}(G)}{\sigma^A_{\min}\eta^2\rho^2} \big)\theta_{t} + \frac{9L^2d^2\mu^2}{\sigma^A_{\min}\rho^2}, \nonumber \\
\end{align}
where the first inequality holds by Lemma \ref{lem:3}.

Next, combining the above inequalities \eqref{eq:A79}, \eqref{eq:A80} and \eqref{eq:A81}, we have
\begin{align}
\min_{1\leq t \leq T} \mathbb{E}\big[ \mbox{dist}(0,\partial L(x_t,y_{[k]}^{t},\lambda_t))^2\big] &\leq \frac{1}{T}\sum_{t=1}^T\mathbb{E}\big[ \mbox{dist}(0,\partial L(x_t,y_{[k]}^{t},\lambda_t))^2\big] \nonumber \\
& \leq \frac{\nu_{\max}}{T} \sum_{t=1}^T \theta_t + \frac{3L^2 d^2 \mu^2}{2} + \frac{9L^2d^2\mu^2}{\sigma^A_{\min}\rho^2} \nonumber \\
& \leq \frac{2\nu_{\max}(\Omega_0 - \Omega^*)}{\gamma T}  + \frac{18\nu_{\max}L^2d^2\mu^2}{ \gamma \sigma^A_{\min}\rho} + \frac{\nu_{\max}L d^2 \mu^2}{2\gamma}
+ \frac{3L^2 d^2 \mu^2}{2} + \frac{9L^2d^2\mu^2}{\sigma^A_{\min}\rho^2} \nonumber \\
& = \frac{2\nu_{\max}(\Omega_0 - \Omega^*)}{\gamma T}  + \big( \frac{18\nu_{\max}L}{ \gamma \sigma^A_{\min}\rho} + \frac{\nu_{\max}}{2\gamma}
+ \frac{3L}{2} + \frac{9L}{\sigma^A_{\min}\rho^2} \big) Ld^2\mu^2\nonumber \\
\end{align}
where the third inequality holds by Lemma \ref{lem:4}, and
 $\nu_{\max}= \max(\nu_1,\nu_2,\nu_3)$, $\gamma = \min(\sigma_{\min}^H, \chi_{t}, L)$, and $\chi_t \geq \frac{3\sqrt{791}\kappa_Gd^l}{2\alpha}\ (l=0,0.5,1)$.

Given $\eta = \frac{\alpha\sigma_{\min}(G)}{33d^lL} \ (0 < \alpha \leq 1, \ l=0,0.5,1 )$ and $\rho = \frac{6\sqrt{791}\kappa_G Ld^l}{\sigma^A_{\min}\alpha}$,
since $k$ is relatively small, it is easy verifies that $\gamma = O(1) $ and $\nu_{\max} = O(d^{2l})$,
which are independent on $n$ and $d$.
Thus, we obtain
\begin{align}
\min_{1\leq t \leq T} \mathbb{E}\big[ \mbox{dist}(0,\partial L(x_t,y_{[k]}^{t},\lambda_t))^2\big]  \leq O(\frac{d^{2l}}{T})  + O(d^{2+2l}\mu^2).
\end{align}

\end{proof}

\end{appendices}

\end{onecolumn}

\end{document}